\documentclass[11pt]{article}

\usepackage[utf8]{inputenc}
\usepackage{amsmath}
\usepackage{amsthm, amssymb}
\usepackage[colorlinks=true, pdfstartview=FitV, linkcolor=blue, citecolor=blue, urlcolor=blue]{hyperref}
\usepackage[colorinlistoftodos]{todonotes}
\usepackage{algorithm}
\usepackage{algpseudocode}
\usepackage{pgf}
\usepackage{tikz}
\usepackage{tikz-cd}
\usetikzlibrary{positioning,shapes,shadows,arrows}
\usepackage{bbm,bm}
\usepackage{enumitem}
\usepackage{mathabx}
\usepackage{comment}
\usepackage{ytableau}
\usepackage{young}

\theoremstyle{definition}
\newtheorem{prop}{Proposition}[section]

\theoremstyle{definition}
\newtheorem{thm}[prop]{Theorem}

\theoremstyle{definition}
\newtheorem{lem}[prop]{Lemma}

\theoremstyle{definition}
\newtheorem{lemma}[prop]{Lemma}

\theoremstyle{definition}
\newtheorem{cor}[prop]{Corollary}

\theoremstyle{definition}
\newtheorem{conj}[prop]{Conjecture}

\theoremstyle{definition}
\newtheorem{exa}[prop]{Example}

\theoremstyle{definition}
\newtheorem{rem}[prop]{Remark}

\theoremstyle{definition}
\newtheorem{defn}[prop]{Definition}

\theoremstyle{definition}

\newcommand{\concat}{\mathsf{concat}}
\newcommand{\sgn}{\operatorname{sgn}}

\newcommand{\NN}{\mathbb{N}}

\newcommand{\rect}{\mathsf{rect}}

\newcommand{\rev}{\mathsf{rev}}

\newcommand{\wt}{\mathsf{wt}}
\newcommand{\bumpwt}{\mathsf{wt}^{\mathsf{bump}}}
\newcommand{\crosswt}{\mathsf{wt}^{\mathsf{cross}}}
\newcommand{\bumpchi}{\chi^{\mathsf{bump}}}
\newcommand{\crosschi}{\chi^{\mathsf{cross}}}

\newcommand{\Z}{\mathbb{Z}}
\newcommand{\ZZ}{\mathbb{Z}}

\newcommand{\fG}{\mathfrak{G}}
\newcommand{\fL}{\mathfrak{L}}

\newcommand{\RPD}{\mathsf{RPD}}

\newcommand{\cC}{\mathcal{C}}
\newcommand{\cD}{\mathcal{D}}
\newcommand{\cP}{\mathcal{P}}

\newcommand{\BB}{\mathbb{B}}
\renewcommand{\SS}{\mathbb{S}}
\newcommand{\One}{1\hspace{-1.25mm}1}

\newcommand{\weight}{\mathrm{wt}}
\newcommand{\quand}{\quad\text{and}\quad}
\newcommand{\qquand}{\qquad\text{and}\qquad}
\newcommand{\e}{\mathbf{e}}

\newcommand{\insss}[2]{(#1 \xrightarrow{H} #2)}
\newcommand{\Tab}{\mathsf{tab}}
\newcommand{\TT}{P_{\mathsf{hecke}}}
\newcommand{\TQ}{Q_{\mathsf{hecke}}}

\newcommand{\HW}{\mathsf{HW}}
\newcommand{\ch}{\mathsf{ch}}
\newcommand{\ben}{\begin{enumerate}}
\newcommand{\een}{\end{enumerate}}
\newcommand{\be}{\begin{equation}}
\newcommand{\ee}{\end{enumerate}}
\newcommand{\PP}{\Z_{>0}}

\newcommand{\col}{\mathsf{col}}
\newcommand{\row}{\mathsf{row}}
\newcommand{\revrow}{\mathsf{revrow}}

\newcommand{\SD}{\mathsf{SD}}
\newcommand{\D}{\mathsf{D}}
\newcommand{\cH}{\mathcal{H}}
\newcommand{\DHF}{\mathsf{Decr}}

\newcommand{\svword}{\mathsf{svword}}

\newcommand{\ytabb}[1]{
\ytableausetup{boxsize = 0.5cm,aligntableaux=center}
{\begin{ytableau}  #1  \end{ytableau}}
}

\newcommand{\ptile}{
\begin{tikzpicture}[x=1em,y=1em,thick,color = blue]
\draw[step=1,gray,thin] (0,0) grid (1,1);
\draw[color=black, thick, sharp corners] (0,0) rectangle (1,1);
\draw(0.5,1.0)--(0.5,0.0);
\draw(1.0,0.5)--(0.0,0.5);
\end{tikzpicture}}

\newcommand{\bumptile}{
\begin{tikzpicture}[x=1em,y=1em,thick,rounded corners,color = blue]
\draw[step=1,gray,thin] (0,0) grid (1,1);
\draw[color=black, thick, sharp corners] (0,0) rectangle (1,1);
\draw(1.0,0.5)--(0.5,0.5)--(0.5,0.0);
\draw(0.5,1.0)--(0.5,0.5)--(0.0,0.5);
\end{tikzpicture}}

\newcommand{\ytab}[1]{
\ytableausetup{boxsize = .4cm,aligntableaux=center}
{\begin{ytableau}  #1  \end{ytableau}}
}

\definecolor{darkblue}{rgb}{0.0,0,0.7} 
\definecolor{darkred}{rgb}{0.7,0,0} 
\definecolor{darkgreen}{rgb}{0, .6, 0} 
\newcommand{\definition}[1]{{\color{darkred}\emph{#1}}} 

\newcommand{\cB}{\mathcal{B}}

\newcommand{\sqgln}{\sqrt{\mathfrak{gl}_n}}

\newcommand{\GP}{G\hspace{-0.2mm}P}
\newcommand{\GPdec}{\GP^{[\mathsf{dec}]}}

\newcommand{\ba}{\begin{aligned}}
\newcommand{\ea}{\end{aligned}}
\newcommand{\barr}{\begin{array}}
\newcommand{\earr}{\end{array}}

\def\gl{\mathfrak{gl}}
\def\q{\mathfrak{q}}

\newcommand{\SVT}{\mathsf{SetTab}}
\newcommand{\SVDT}{\mathsf{SetDecTab}}

\newcommand{\SVWordsss}[2]{\mathsf{SetWord}_{#2,#1}}

\newcommand{\IncT}[2]{\mathsf{IncT}_{#2,#1}}

\newcommand{\minlevel}{\operatorname{minlevel}}

\def\gamma{\svword}
\def\fkD{{\mathfrak D}}
\usepackage{fullpage}

\numberwithin{equation}{section}

\begin{document}

\title{Grothendieck positivity for normal square root crystals}
\author{
    Eric MARBERG
    \\
    HKUST \\
    {\tt emarberg@ust.hk}
    \and
    Kam Hung TONG\\
    Hong Kong Polytechnic University \\
    {\tt kam-hung-terry.tong@polyu.edu.hk}
    \and
    Tianyi YU 
    \\
    Universit\'e du Qu\'ebec \`a Montr\'eal\\
    {\tt yu.tianyi@uqam.ca}
}

\date{}

\maketitle

\begin{abstract} 
Normal crystals (also known as Stembridge crystals) are commonly used to establish the Schur positivity of symmetric functions, as their characters are sums of Schur polynomials. In this paper, we develop a combinatorial framework for a novel family of objects called normal square root crystals, which are closely related to symmetric Grothendieck functions, the $K$-theoretic analogue of Schur functions. Among other applications, this tool leads to a new proof of Buch's combinatorial rule for the multiplication of symmetric Grothendieck functions. The definition of a normal square root crystal, originally formulated by the first two authors, largely mirrors that of normal crystals. Our main result is to show that the character of such a crystal is always a sum of symmetric Grothendieck polynomials. The proof relies on an unexpected connection between the raising operators for our crystals and the Hecke insertion algorithm developed by Buch, Kresch, Shimozono, Tamvakis, and Yong.
\end{abstract}

\tableofcontents

\section{Introduction}

\subsection{Overview}

Crystals are combinatorial objects that arise in the representation theory of quantum groups. They were first studied in independent work of Kashiwara \cite{Kashiwara1990,Kashiwara1991} and Lusztig \cite{Lusztig1990a,Lusztig1990b}; for a relevant history, see \cite[\S1]{BumpSchilling}.
A theory of abstract crystals exists for any finite-dimensional Lie algebra,
but in this paper we focus on two families of crystals associated to type $\gl_n$ for a fixed positive integer $n$.

In our setting, a \definition{crystal} consists of a directed graph $\cB$ with edges labeled by indices in the set $\{1,2,\dots,n-1\}$,
along with a weight map $\weight : \cB \to \ZZ^n$ assigning an integer $n$-tuple to each vertex. This data must obey some additional axioms, as explained in Section~\ref{crystal-sect}.
When the vertex set $\cB$ 
is finite, the crystal has a polynomial \definition{character} provided by the weight-generating function
\begin{equation} \textstyle \ch(\cB) := \sum_{b \in \cB} x^{\weight(b)} \in \Z[x_1,x_2,\dots,x_n]
\quad\text{where $x^{(w_1,w_2,\dots,w_n)} := x_1^{w_1} x_2^{w_2}\cdots x_n^{w_n}$.}
\end{equation}
Isomorphisms between crystals correspond to graph isomorphisms that preserve edge labels and vertex weights. Isomorphic crystals therefore have the same character, but the converse is not always true.

As mentioned above, we will consider two families of crystals here, which obey slightly different axioms. 
Each family has a distinguished \definition{standard crystal}.
Moreover, the two families are equipped with the same \definition{tensor product $\otimes$}, an associative operator that 
takes two given crystals $\cB$ and $\cC$
and forms a new crystal $\cB\otimes \cC$ in the same family.
This operator satisfies $\ch(\cB\otimes \cC) = \ch(\cB)\ch(\cC)$ when $\cB$ and $\cC$ are finite.

Our first family consists of all \definition{(seminormal) $\gl_n$-crystals}, following the conventions in \cite{BumpSchilling}. For this well-studied type of crystal, the standard object (which we call the \definition{standard $\gl_n$-crystal})
is the path graph presented in Figure~\ref{gl-fig}. This crystal corresponds to the vector representation of the quantum group $U_q(\gl_n)$, and its vertices are labeled by the numbers $1, 2,\dots, n$.
Our second family is made up of the \definition{square root crystals} (abbreviated as \definition{$\sqgln$-crystals})
introduced much more recently in \cite{MT,Y}. For this type, the standard object (which we call
the \definition{standard $\sqgln$-crystal}),
is a certain graph on all non-empty subsets of $[n] :=\{1,2,\dots,n\}$,
as shown in Figure~\ref{SS3-fig} for the case $n=3$.   
For the precise definitions of these crystals and their tensor product; see Section~\ref{crystal-sect}.

 We are specifically interested in the \definition{normal} 
 crystals within these two families.
A $\gl_n$-crystal (respectively, $\sqgln$-crystal) is \definition{normal} if it is isomorphic to a disjoint union of connected components of tensor powers of the standard $\gl_n$-crystal (respectively, the standard $\sqgln$-crystal). 
Normal $\gl_n$-crystals have many remarkable features. They precisely correspond to crystal bases of quantum group representations \cite{BumpSchilling}. They can also be characterized by a 
 set of local \definition{Stembridge axioms} \cite{Stembridge} and for this reason are sometimes called \definition{Stembridge crystals}.
 
In addition, there is a close connection between normal $\gl_n$-crystals and positivity properties of certain symmetric functions.
Our main result will demonstrate a similar phenomenon for normal $\sqgln$-crystals.
To explain this, recall that
the subring of symmetric elements in $\ZZ[x_1,x_2,\dots,x_n]$
has two well-known $\ZZ$-bases given by the 
\definition{Schur polynomials} $s_\lambda(x_1,x_2,\dots,x_n)$
and \definition{(signless) symmetric Grothendieck polynomials} $G_\lambda(x_1,x_2,\dots,x_n)$.
These bases are both indexed by integer partitions $\lambda = (\lambda_1\geq \lambda_2 \geq \dots \geq \lambda_n\geq 0)$ with at most $n$ parts; see Section~\ref{sym-intro} for the definitions.

In the literature, the terms ``symmetric Grothendieck polynomial'' and ``stable Grothendieck polynomial'' often refer to the function
$\beta^{-|\lambda|}G_\lambda(\beta x_1,\beta x_2,\dots,\beta x_n)$ where $\beta$ is an extra parameter. 
When one sets $\beta=-1$,
these polynomials are representatives for the $K$-theory classes of Schubert varieties in the complex Grassmannian \cite[Thm.~8.1]{Buch2002}, while Schur polynomials are representatives for the cohomology classes.
Schur polynomials and symmetric Grothendieck polynomials have many other interpretations in representation theory and geometry. 

Because of these interpretations, it is often meaningful to prove that generating functions are \definition{Schur positive} (respectively, \definition{$G$-positive}) in the sense of being equal to a $\NN$-linear combination of Schur polynomials (respectively, symmetric Grothendieck polynomials).
When this property holds, one is further interested in finding positive combinatorial formulas for the coefficients.

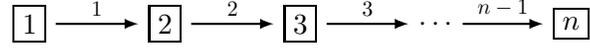
\begin{figure}
\[
    \begin{tikzpicture}[xscale=1.8, yscale=1,>=latex,baseline=(z.base)]
    \node at (0,0.0) (z) {};
      \node at (0,0) (T0) {$\boxed{1}$};
      \node at (1,0) (T1) {$\boxed{2}$};
      \node at (2,0) (T2) {$\boxed{3}$};
      \node at (3,0) (T3) {${\cdots}$};
      \node at (4,0) (T4) {$\boxed{n}$};
      \draw[->,thick]  (T0) -- (T1) node[midway,above,scale=0.75] {$1$};
      \draw[->,thick]  (T1) -- (T2) node[midway,above,scale=0.75] {$2$};
      \draw[->,thick]  (T2) -- (T3) node[midway,above,scale=0.75] {$3$};
      \draw[->,thick]  (T3) -- (T4) node[midway,above,scale=0.75] {$n-1$};
     \end{tikzpicture}
\]
\caption{The \definition{standard $\gl_n$-crystal}, 
for which the weight map is $\weight(\boxed{i})=\e_i\in\ZZ^n$.
}\label{gl-fig}
\end{figure}

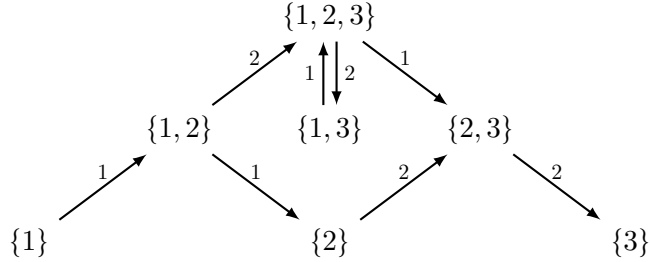
\begin{figure}[h]
\centerline{
\begin{tikzpicture}[xscale=2, yscale=0.75,>=latex,baseline=(z.base)]
    \node at (0,0.0) (z) {};
      \node at (0,0) (T1) {$\{1\}$};
      \node at (1,2) (T12) {$\{1,2\}$};
      \node at (2,0) (T2) {$\{2\}$};
      \node at (3,2) (T23) {$\{2,3\}$};
      \node at (4,0) (T3) {$\{3\}$};
      \node at (2,4) (T123) {$\{1,2,3\}$};
      \node at (2,2) (T13) {$\{1,3\}$};
      \draw[->,thick,blue]  (T1) -- (T12) node[midway,above,scale=0.75] {$1$};
      \draw[->,thick,blue]  (T12) -- (T2) node[midway,above,scale=0.75] {$1$};
      \draw[->,thick,red]  (T2) -- (T23) node[midway,above,scale=0.75] {$2$};
      \draw[->,thick,red]  (T23) -- (T3) node[midway,above,scale=0.75] {$2$};
      \draw[->,thick,red]  (T12) -- (T123) node[midway,above,scale=0.75] {$2$};
      \draw[->,thick,red]  (T123.285) -- (T13.75) node[midway,right,scale=0.75] {$2$};
      \draw[->,thick,blue]  (T13.105) -- (T123.255) node[midway,left,scale=0.75] {$1$};
      \draw[->,thick,blue]  (T123) -- (T23) node[midway,above,scale=0.75] {$1$};
     \end{tikzpicture}
}
     \caption{The standard $\sqrt{\gl_3}$-crystal, for which the weight map is $\weight(S) = \sum_{i \in S} \e_i \in \ZZ^3$.}\label{SS3-fig}
     \end{figure}

An element of $\gl_n$- or $\sqgln$-crystal is \definition{highest weight} if it is a source vertex in the associated crystal graph. 
Let $\HW(\cB)$ denote the set of highest weight elements in a crystal $\cB$.
When $\cB$ is normal, the weight of every 
element of $\HW(b)$ is guaranteed to be a partition with at most $n$ parts; see \cite[\S2.4]{BumpSchilling} for the $\gl_n$-case and \cite[Lem.~4.11]{MT} for the $\sqgln$-case.

It is known \cite[Thms.~3.2 and 8.6]{BumpSchilling}
that if $\cB$ is a finite normal $\gl_n$-crystal, then
\begin{equation}
    \textstyle \ch(\cB) = \sum_{b \in \HW(\cB)} s_{\weight(b)}(x_1,x_2,\dots,x_n).
\end{equation}
Via this identity, normal $\gl_n$
crystals can be a useful tool for demonstrating Schur positivity: given a generating function over some discrete set of objects, one just needs to give this set an appropriate crystal structure.
For examples of this approach, see \cite{HawScr,MorseSchilling}.

Our main result demonstrates a similar property for the characters of normal $\sqgln$-crystals, which was predicted by the first two authors as \cite[Conj.~4.37]{MT}:

\begin{thm}\label{ch-thm}
If $\cB$ is a finite normal $\sqgln$-crystal
then 
$\textstyle \ch(\cB) = \sum_{b \in \HW(\cB)} G_{\weight(b)}(x_1,x_2,\dots,x_n).$
\end{thm}

The proof of this theorem relies on an unexpected connection between highest weight elements in normal square root crystals and the \definition{Hecke insertion algorithm} from \cite{BKSTY}; see
Section~\ref{gp-sect}.

Theorem~\ref{ch-thm} gives a way of using square root crystals to prove Grothendieck positivity properties. The rest of this introduction outlines some applications along these lines. These sample applications include new proofs of known results from \cite{Buch2002,BKSTY}, as well as one new identity.

\subsection{Expanding skew symmetric Grothendieck functions}\label{sym-intro}

Combined with some prior results in \cite{MT,Y},
Theorem~\ref{ch-thm} leads to a quick proof of Buch's combinatorial rule for the positive $G$-expansion of any 
\definition{skew symmetric Grothendieck polynomial}. We recall the definition of this generalization of $G_\lambda(x_1,x_2,\dots,x_n)$ below.

The \definition{Young diagram} of a  partition $\lambda=(\lambda_1\geq \lambda_2\geq \dots \geq 0)$ is the set of boxes 
\begin{equation}\label{D-eq}
    \D_\lambda := \{(i,j) \in \PP\times \PP : 1\leq j \leq \lambda_i\}.
\end{equation}
If $\D_\lambda \subseteq \D_\nu$ 
then one writes 
$\lambda\subseteq \nu$
and defines $\D_{\nu/\lambda} := \D_\nu \setminus \D_\lambda$.

A \definition{set-valued tableaux} of shape $\nu/\lambda$
is a filling of the boxes in $\D_{\nu/\lambda}$ by non-empty subsets of $[n]$.
The \definition{distributions} of such a tableau are formed by replacing each set-valued entry by one of its elements.
A set-valued tableau $T$ is \definition{semistandard} if
all of its distributions have weakly increasing rows and strictly increasing columns.
For example,
\begin{equation}\label{t-ex}
\ytableausetup{boxsize=0.7cm}
T\ =\ \begin{ytableau}
\none & 134 & 45 & 5 \\
45 & 5\\
89
\end{ytableau}
\end{equation}
is a semistandard set-valued tableau of shape $\nu/\lambda=(4,2,1)/(1)$ drawn in English notation.
If $T$ is a set-valued tableau then we write $(i,j) \in T$ to indicate that box $(i,j)$ is filled in $T$, and we let $T_{ij}$ denote the entry in this box. For example, if $T$ is as in \eqref{t-ex} then $T_{22}= \{5\}$ and $T_{31} = \{8,9\}$.

Let $\SVT_n(\nu/\lambda)$ denote the family of all semistandard set-valued tableaux of shape $\nu/\lambda$, and write $\SVT_n(\nu)$ in place of $\SVT_n(\nu/\emptyset)$
Define also the \definition{weight} of $T$ to be 
\begin{equation}
    \label{t-wt}
\textstyle \weight(T) := \sum_{(i,j) \in T} \sum_{k \in T_{ij}} \e_k \in \ZZ^n
\end{equation}
where $\e_1,\e_2,\dots,\e_n$ is the standard basis of $\ZZ^n$. 
The weight of the tableau in \eqref{t-ex} when $n=9$ is $\weight(T) = (1,0,1,3,4,0,0,1,1)$ .

Let $x_1,x_2,\dots,x_n$ be commuting variables and set $x^\alpha= x_1^{\alpha_1}x_2^{\alpha_2}\cdots x_n^{\alpha_n}$ for any tuple $\alpha 
\in \ZZ^n$.

\begin{defn} If $\lambda\subseteq\nu$ are partitions then let
$
G_{\nu/\lambda}(x_1,x_2,\dots,x_n)=
\sum_{T \in \SVT_n(\nu/\lambda)} x^{\weight(T)}.
$\
\end{defn}

When $\lambda=\emptyset$ is the empty partition, the \definition{skew symmetric Grothendieck polynomial}
defined in this way gives $G_{\nu}(x_1,x_2,\dots,x_n)$ from the previous section.
One may then define $s_\nu(x_1,x_2,\dots,x_n)$ as the homogeneous term of $G_\nu(x_1,x_2,\dots,x_n)$ of lowest degree.

If $T$ is a set-valued tableau with $m$ boxes, then its \definition{column reading word}
is the $m$-tuple of sets $\col(T)$
obtained by reading the entries $T$ up each column (in English notation), iterating over the columns from left to right. 
For any tuple of sets of integers $S=(S_1,S_2,\dots,S_m)$, let $w(S)$ be the word 
formed 
by replacing each $S_i$ by its entries in increasing order.
If $T$ is the tableau in \eqref{t-ex} then \[
\col(T) = (\{8,9\}, \{4,5\},\{5\},\{1,3,4\},\{4,5\},\{5\})
\quand
w(\col(T)) = 
89455134455.\]
Finally, following \cite{Buch2002}, we define a positive integer sequence $w=w_1w_2\cdots w_q$ to be a \definition{reverse lattice word} if the integer vector $\weight(w_iw_{i+1}w_{i+2}\cdots w_q)$ is weakly decreasing for each $i \in [q]$,
where
$\weight(w) := \sum_{i=1}^q \e_{w_i}$
The word $89455134455$ does not have this property, but $45132431121$ is an example.

The following result is implicit in \cite{MT,Y}, but we present it here for future reference.

\begin{thm}\label{svt-thm}
The set $\SVT_n(\nu/\lambda)$ has a normal $\sqgln$-crystal structure
    with weight map \eqref{t-wt} such that $T\in\SVT_n(\nu/\lambda)$
    is highest weight if and only if $w(\col(T))$ is a reverse lattice word.
\end{thm}

\begin{proof}
Let $m= |\nu/\lambda| := |\D_{\nu/\lambda}|$.
Via the column reading word, we can interpret each tableau in $\SVT_n(\nu/\lambda)$ as 
an element of the $m$-fold tensor power of the standard $\sqgln$-crystal.
The results in
\cite[\S4]{Y} 
show that this set of elements is preserved by the associated crystal operators, which allows us to transfer a normal $\sqgln$-crystal structure to $\SVT_n(\nu/\lambda)$ with the weight map \eqref{t-wt}.
By \cite[Lem.~4.11 and Thm.~4.17]{MT}, the highest weight elements of this crystal are as described.
\end{proof}

We can now present our first application of Theorem~\ref{ch-thm}.

\begin{cor}[\cite{Buch2002}] 
If $\lambda\subseteq\nu$ are partitions then 
$\textstyle
G_{\nu/\lambda}(x_1,\dots,x_n) = \sum_\mu a^\nu_{\lambda\mu} G_\mu(x_1,\dots,x_n)$
where the sum is over all partitions $\mu$ with at most $n$ parts
and $a^\nu_{\lambda\mu} $ is the number of set-valued tableaux $T\in \SVT_n(\nu/\lambda)$ for which $w(\col(T))$ is a reverse lattice word of weight $\mu$.
\end{cor}

\begin{proof}
Apply Theorem~\ref{ch-thm}
to the normal $\sqgln$-crystal given in Theorem~\ref{svt-thm}.
\end{proof}

\begin{rem}\label{limit-rem}
If one lets  $n\to\infty$, then the polynomials
$G_{\nu/\lambda}(x_1,x_2,\dots,x_n)$
and $G_{\mu}(x_1,x_2,\dots,x_n)$
converge in the sense of formal power series to well-defined symmetric functions $G_{\nu/\lambda}$ and $G_\mu$.
After taking this limit and adjusting signs (the symmetric function written ``$G_\lambda$'' in \cite{Buch2002}
corresponds in our notation to $(-1)^{|\lambda|} G_\lambda(-x)$) the preceding corollary becomes \cite[Thm.~9.3]{Buch2002}.
\end{rem}

\subsection{A new proof of the Littlewood--Richardson rule of Buch}

As another application of Theorem~\ref{ch-thm}, we can give an alternate proof of the \definition{Littlewood--Richardson rule} of Buch that decomposes the product of two symmetric Grothendieck polynomials.

If $\lambda$ is a partition with largest part $p$ and $\mu$ is a partition with $q$ parts, then let $\lambda\ast\mu$
be the skew shape formed by
    shifting the Young diagram of $\lambda$ down by $q$ rows, then shifting the Young diagram of $\mu$ to the right by $p$ columns, and finally taking the union; equivalently:
    \begin{equation}
     \lambda\ast\mu := (p+\mu_1,p+\mu_2,\dots,p+\mu_q,\lambda_1,\lambda_2,\lambda_3,\dots)/(p^q).
 \end{equation}
After letting $n\to \infty$ and adjusting signs
as in Remark~\ref{limit-rem},
 the following recovers \cite[Thm.~5.4]{Buch2002}:

\begin{cor}[\cite{Buch2002}] 
    If $\lambda$ and $\mu$ are partitions then
    \[\textstyle G_\lambda(x_1,x_2,\dots,x_n) G_\mu(x_1,x_2,\dots,x_n)= \sum_{\nu} c_{\lambda\mu}^\nu G_\nu(x_1,x_2,\dots,x_n)\]
    where the sum is over all partitions $\nu$ with at most $n$ parts 
    and $c_{\lambda\mu}^\nu$ is the number of set-valued tableaux $T\in \SVT_n(\lambda \ast \mu)$ for which $w(\col(T))$ is a reverse lattice word of weight $\nu$.
\end{cor}

\begin{proof}
For the $\sqgln$-crystal structure
 in Theorem~\ref{svt-thm} and the tensor product defined in Section~\ref{crystal-sect}, it holds by construction that 
$\SVT_n(\lambda) \otimes \SVT_n(\mu) \cong \SVT_n(\lambda\ast\mu)$.
Given this observation and Theorem~\ref{ch-thm},
both sides of the desired identity are the character of $\SVT_n(\lambda\ast\mu)$.
\end{proof}

\subsection{Symmetric Grothendieck functions of permutations}\label{pi-sect}

There is a generalization of $G_\lambda(x_1,x_2,\dots,x_n)$ indexed by permutations rather than partitions.
These polynomials arise as the stable limits of $K$-theory representatives for \definition{Schubert varieties} in the complete flag variety \cite{LS82,FL}. They were introduced in \cite{FK} and are known to be $G$-positive \cite{BKSTY}.
Here, we outline how to derive this positivity property from Theorem~\ref{ch-thm}.

Let $S_\infty$ be the group of permutations of $\PP$ with finite support. For each $m \in \PP$ we view $S_m \subset S_\infty$ as the finite subgroup of permutations fixing all points outside $[m]$.
The \definition{Demazure product} on $S_\infty$ is the unique associative binary operation $\circ : S_\infty\times S_\infty \to S_\infty$
that satisfies 
$s_i\circ s_i = s_i$ for $s_i:=(i,i+1) \in S_\infty$
and $w\circ s_i = ws_i$ for $w \in S_\infty$ with $w(i)<w(i+1)$.
 A \definition{Hecke word} for $w \in S_\infty$ is a finite sequence of positive integers $i_1i_2\cdots i_k$ with $w=s_{i_1} \circ s_{i_2}\circ \cdots \circ s_{i_k}.$
 
 Let $\cH(w)$ be the set of all Hecke words for $w$. 
As explained in \cite{BKSTY}, each set $\cH(w)$ is an equivalence class  for the transitive closure of the symmetric relation on integer sequences with
\begin{equation}\label{hecke-rel}
  \ba  \cdots p\cdots &\sim \cdots pp\cdots &&\text{for all $p \in \ZZ$} \\
  \cdots pqp\cdots &\sim \cdots qpq\cdots &&\text{for $p \neq q$} \\
  \cdots pq\cdots &\sim \cdots qp\cdots &&\text{for $|p-q|>1$.}
  \ea
  \end{equation}
 Here, the subwords masked by corresponding ellipses ``$\cdots$'' are required to be identical.

 A \definition{decreasing Hecke factorization} of $w \in S_\infty$ is an $n$-tuple $a=(a^1, a^2,\dots,a^n)$
of strictly decreasing words $a^i$ such that the concatenation $\concat(a):= a^1a^2 \cdots a^n$ belongs to $\cH(w)$. 
 Let 
 \begin{equation}\label{wt-eq2} \weight(a) := (\ell(a^1),\ell(a^2),\dots,\ell(a^n)) \in \ZZ^n\end{equation}
and write $\DHF_n(w)$
for the set of all decreasing Hecke factorizations of $w \in S_\infty$.

\begin{defn} For each permutation $w \in S_\infty$ let
$G_w(x_1,x_2,\dots,x_n) = \sum_{a \in \DHF_n(w)} x^{\weight(a)}$.
\end{defn}

\begin{rem}\label{beta-rem} This definition differs slightly from prior literature by omitting signs or more general bookkeeping parameters. These can be recovered
by making appropriate variable substitutions. 
For example, what is denoted ``$G_w$'' in \cite{BKSTY}
is expressed in our notation as $\sgn(w) G_w(-x)$.
\end{rem}

  \begin{exa}\label{pi-ex}
 The Hecke words of $w=1432\in S_4$ 
 consist of all words in the alphabet $\{2,3\}$ 
 except those
 in which every 2 precedes every 3, or vice versa.
 The set
  $
  \DHF_2(w) = 
  \{ (2,32), (32,3), (32,32)\} $ 
  is finite,
  and we have 
  $G_{w}(x_1,x_2) = x_1x_2^2 + x_1^2 x_2 + x_1^2x_2^2 = G_{\lambda}(x_1,x_2)$ for $\lambda=(2,1)$.
 \end{exa}

An \definition{increasing tableau} is a filling of the Young diagram of a  partition by positive integers such that all rows and columns are strictly increasing.
Given $a=(a^1,a^2,\dots,a^n) \in \DHF_n(w)$,
let $\Tab(a)$ be tableau
formed by placing the reversal of $a^i$ in row $i$.
For example, if $a=(32,3)$ then $\Tab(a) = \ytab{2 & 3 \\ 3}$. In general, this tableau need not be increasing or have partition shape. However:

\begin{thm}\label{dhf-thm}
The set $\DHF_n(w)$ has a normal $\sqgln$-crystal structure
    with weight map \eqref{wt-eq2} such that $a \in \DHF_n(w)$
    is highest weight if and only if $\Tab(a)$ is an increasing tableau.
\end{thm}

We defer the proof of this theorem to Section~\ref{Tab-sect},
where we can give a succinct argument using some of the lemmas derived on the way to Theorem~\ref{ch-thm}.

We have already defined the column reading word of a (set-valued) tableau.
The \definition{row reading word} of a tableau $T$ is 
the word $\row(T)$ formed by reading the rows from left to right, starting with the bottom row in English notation.
The \definition{reverse reading word} $\revrow(T)$ is the reversal
of $\row(T)$.

\begin{cor}[\cite{BKSTY}]
\label{bksty-cor}
    If $w \in S_\infty$ then 
    $\textstyle 
    G_{w}(x_1,x_2,\dots,x_n) = \sum_\lambda c_{w\lambda} G_\lambda(x_1,x_2,\dots,x_n)$
    where the sum is over all partitions $\lambda$ with at most $n$ parts and 
    $c_{w\lambda}$ is the number of increasing tableaux $T$ of shape $\lambda$ with $\revrow(T)\in \cH(w)$.
\end{cor}

\begin{proof}
As $\revrow(\Tab(a)) = \concat(a)$, 
it follows from Theorem~\ref{dhf-thm} that $c_{w\lambda}$ counts the highest weight elements in $\DHF_n(w)$ of weight $\lambda$.  
The corollary therefore holds by Theorem~\ref{ch-thm}. 
\end{proof}

The reading words $\row(T)$ and $\col(T)$ are equivalent under the relation \eqref{hecke-rel} when $T$ is increasing  \cite[Lem.~2.7]{Mar2019}, and it holds that $\revrow(T) \in \cH(w)$ if and only if
$\row(T) \in \cH(w^{-1})$.
Thus, one could also define $c_{w\lambda}$ in Corollary~\ref{bksty-cor} as 
the number of increasing tableaux $T$ of shape $\lambda$ with $\col(T)\in \cH(w^{-1})$.
This formulation of Corollary~\ref{bksty-cor} becomes  
\cite[Thm.~1]{BKSTY}
after taking the power series limit $n\to \infty$ and adjusting signs as in Remark~\ref{limit-rem}.

\subsection{Generating functions of set-valued decomposition tableaux}

Finally, we explain one new identity that can be proved using Theorem~\ref{ch-thm}.
Assume $\lambda = (\lambda_1 > \lambda_2 > \dots \geq 0)$ is a \definition{strict} integer partition.
A \definition{shifted tableau} of shape $\lambda$ is a filling of the \definition{shifted diagram} $\SD_\lambda = \{ (i,i+j-1) : (i,j ) \in \D_\lambda\}$
by numbers in $[n]$.

Let $T$ be a shifted tableau whose rows read left-to-right 
are \definition{hook words}, meaning integer sequences $w_1w_2\cdots w_N$ 
with $ w_1 \geq w_2 \geq \dots \geq w_m < w_{m+1} < w_{m+2} <\dots <w_N$ for some $1\leq m \leq N$.
Following \cite{GJKKK,MT}, we define $T$ to be a   \definition{(semistandard) decomposition tableau} if none of the following patterns occur in consecutive rows: 
\[
 \ytabb{
 \none & a  & \cdots & \  \\
 \none & \none & \cdots & b  
 }
\quad
 \ytabb{
\ & \cdots & a & \cdots & \  \\
\none & \cdots  & c & \cdots & b  \\
 }
 \quad\text{for $a\leq b\leq c$}
  \quad\text{or}\quad
 \ytabb{
y & \cdots & z  \\
 \none & \cdots  &x     
 }
 \quad
 \ytabb{
\ & \cdots & y & \cdots & z \\
 \none & \cdots  &  & \cdots & x  
 }\quad\text{for }x<y<z.
\]
The tableaux here are drawn in English notation. The leftmost boxes are on the main diagonal and the boxes $\ytabb{\cdots}$ with ellipses indicate zero or more intervening columns.

 A \definition{set-valued decomposition tableau} of shape $\lambda$ is a filling of $\SD_\lambda$ by non-empty subsets of $[n]$ whose distributions are all decomposition tableaux. An example of shape $\lambda=(3,2)$ for $n=5$ is
 \[
 \ytableausetup{boxsize = 0.7cm}
T=\begin{ytableau}
45 & 3 & 234  \\
\none & 12 & 3
\end{ytableau}.
 \]
 Let $\SVDT_n(\lambda)$ be the family of all set-valued decomposition tableaux of shape $\lambda$.
 We define the \definition{reverse row reading word} 
of such a tableau $T$
to be the sequence of sets $\revrow(T)$ formed 
by reading its rows from right to left, starting with the top row in English notation.
Also let $\weight(T)$ be as in \eqref{t-wt}.
 Our example has $\revrow(T) = (\{2,3,4\},\{3\},\{4,5\},\{3\},\{1,2\})$
and $\weight(T) =(1,2,3,2,1)$.

\begin{defn}
For each strict $\lambda$ partition let
$
\GPdec_\lambda(x_1,x_2,\dots,x_n) := \sum_{T \in \SVDT_n(\lambda)} x^{\weight(T)}.
$
\end{defn}
Cho and Ikeda have conjectured  \cite[Conj.~3.2]{MT} that this polynomial coincides with the \definition{$K$-theoretic Schur $P$-function} $\GP_\lambda(x_1,x_2,\dots,x_n)$ introduced in \cite{IkedaNaruse}, which is the weight-generating function for a different family of \definition{(semistandard) set-valued marked shifted tableaux}.

It was shown in \cite[Cor.~3.11]{MT} that  $\GPdec_\lambda(x_1,\dots,x_n) $ at least has the unitriangular form $\GP_\lambda(x_1,\dots,x_n) + \sum_{|\mu|>|\lambda|} \ZZ \GP_\mu(x_1,\dots,x_n)$.
Here, we prove that $\GPdec_\lambda(x_1,\dots,x_n) $ is $G$-positive:

\begin{cor}\label{GPdec-cor}
    If $\lambda$ is a strict partition then 
   $ \textstyle\GPdec_\lambda(x_1,x_2,\dots,x_n) = \sum_\mu g_{\lambda\mu} G_\mu(x_1,x_2,\dots,x_n)$
    where the sum is over all partitions $\mu$ with at most $n$ parts and $g_{\lambda\mu}$ is the number of tableaux $T \in \SVDT_n(\lambda)$ for which $w(\revrow(T))$ is a reverse lattice word of weight $\mu$.
\end{cor}

\begin{proof}
Prior work of the first two authors \cite[Thm.~4.18]{MT}
shows that $\SVDT_n(\lambda)$
    has a normal $\sqgln$-crystal structure 
    with weight map \eqref{t-wt} such that $T\in\SVDT_n(\lambda)$
    is highest weight if and only if $w(\revrow(T))$ is a reverse lattice word.
    Apply Theorem~\ref{ch-thm} to this crystal.
\end{proof}

We mention that it is known \cite[Thm.~3.27]{MaSc} that the polynomial $\GP_\lambda(x_1,x_2,\dots,x_n)$ is $G$-positive. Thus,
one approach to proving \cite[Conj.~3.2]{MT} would be to show that the relevant $G$-expansion
of
$\GP_\lambda(x_1,x_2,\dots,x_n)$
is also $\sum_\mu g_{\lambda\mu} G_\mu(x_1,x_2,\dots,x_n)$.

\begin{exa}
We can use Corollary~\ref{GPdec-cor} to  decompose $\GPdec_\lambda(x_1,\dots,x_n)$  when $\lambda = (m)$ has one part. Let $k = \max\{1, m-n+1\}$. One can show as an exercise from \cite[\S4]{MT} that the highest weight elements in the $\sqgln$-crystal $\SVDT_n((m))$ are the tableaux of shape $(m)$ of the form
\[
\ytableausetup{boxsize = 0.6cm} \begin{ytableau}
    1 & 1 & \cdots & 1 & 2 & 3& \cdots & j
\end{ytableau}
\quand
\ytableausetup{boxsize = 0.6cm} \begin{ytableau}
    1 & 1 & \cdots & 1 & 12 & 3& 4 &  \cdots & j
\end{ytableau}
\quad\text{ with }j\leq n.\]
It then follows from Corollary~\ref{GPdec-cor}  that 
\begin{equation}
\textstyle
    \GPdec_{(m)}(x_1,\dots,x_n) = \sum_{i = k}^m G_{(i,1^{m-i})}(x_1,\dots,x_n) + \sum_{i = k+1}^m G_{(i, 1^{m+1-i})}(x_1,\dots,x_n).
\end{equation}
We know of at least one proof that $\GP_{(m)}(x_1,\dots,x_n)$ has the same $G$-expansion (see Figure~\ref{GP-fig} for a bijective argument), so this verifies \cite[Conj.~3.2]{MT} for one-row strict partitions.

\end{exa}

\newcommand{\downmapsto}{\rotatebox[origin=c]{-90}{$\mapsto$}\mkern2mu}

\begin{figure}[h]
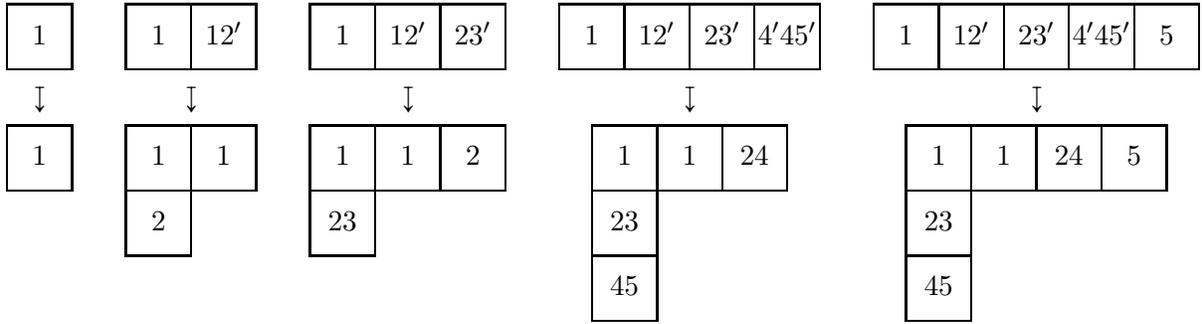

\[
\barr{ccccccccc}
\ytableausetup{aligntableaux=top,boxsize = 0.85cm}
\begin{ytableau}
    1  
\end{ytableau}
&&
\begin{ytableau}
    1 & 12' 
\end{ytableau}
&&
\begin{ytableau}
    1 & 12' & 23'
\end{ytableau}
&&
\begin{ytableau}
    1 & 12' & 23' & 4'45' 
\end{ytableau}
&&
\begin{ytableau}
    1 & 12' & 23' & 4'45' & 5
\end{ytableau}
\\
\downmapsto && \downmapsto && \downmapsto &&\downmapsto && \downmapsto
\\[-10pt]\\
\begin{ytableau}
    1
\end{ytableau}
&&
\begin{ytableau}
    1 & 1\\
    2
\end{ytableau}
&&
\begin{ytableau}
    1 & 1 & 2\\
    23
\end{ytableau}
&&
\begin{ytableau}
    1 & 1 & 24\\
    23 \\
    45
\end{ytableau}
&&
\begin{ytableau}
    1 & 1 & 24 & 5\\
    23 \\
    45
\end{ytableau}
\earr
\]
\caption{
The polynomial $\GP_{(m)}(x_1,\dots,x_n)$ is the weight-generating function for all fillings $T$ of a one-row Young diagram with $m$ boxes  
by nonempty subsets of $\{1'<1<2'<2<\dots<n'<n\}$
such that $\max(T_{1j}) \leq \min(T_{1,j+1})$ for all $j \in[m-1]$, with strict inequality if $\max(T_{1j})$ is primed \cite[\S9.1]{IkedaNaruse}.
One can inductively construct a weight-preserving bijection between such tableaux and the union
 $\bigsqcup_{i=k}^m \SVT_{(i, 1^{m-i})}\sqcup  \bigsqcup_{i=k+1}^m \SVT_{(i, 1^{m+1-i})}$
for $k=\max\{1,m-n-1\}$.
Rather than presenting a formal definition, we show an  example of this bijection to illustrate the main idea. 
}\label{GP-fig}
\end{figure}

\subsection{Future directions and outline}

This work is not the end of the story for square root crystals. It would be very interesting to find connections between normal square root crystals and representation theory. Another open problem is to determine if there are local axioms in the style of \cite{GHPS,Stembridge} that classify which square root crystals are normal. 

We also expect that there are interesting versions of square root crystals for other types. 
There is at least a theory of (normal) square root crystals \cite{MT}  associated to the queer Lie superalgebra $\q_n$. Such crystals have the same relationship to the \definition{$\q_n$-crystals} introduced in work of Grantcharov et al. \cite{GJKKK} as $\sqgln$-crystals do to $\gl_n$-crystals. The first two authors conjectured a ``shifted'' analogue of Theorem~\ref{ch-thm} for these objects, in which $G$-positivity is replaced by  \definition{$\GP$-positivity}; see \cite[Conj.~4.36]{MT}.

A brief outline of the rest of this article is as follows.
Section~\ref{prelim-sect}
reviews the formal definitions of $\gl_n$- and $\sqgln$-crystals, and then proves some new results about highest weight elements in normal $\sqgln$-crystals.
Section~\ref{gp-sect}
contains the proofs of Theorems~\ref{ch-thm} and \ref{dhf-thm}. These proofs are based on a nontrivial reformulation of the \definition{Hecke insertion algorithm} from \cite{BKSTY}, which is explained in Section~\ref{hecke-sect}.
Section 4, finally, discusses a square root analogue of \definition{Demazure crystals}
together with a positivity conjecture  
about the characters of these objects.

\subsection*{Acknowledgments}

This work was partially supported by Hong Kong RGC grant 16304122. 
We thank Valentin Buciumas, Zachary Hamaker, Takeshi Ikeda, Joel Lewis, Brendan Pawlowski, Travis Scrimshaw, Mark Shimozono, and Alexander Yong  
for useful conversations.

\section{Square root crystals}\label{prelim-sect}

Throughout, $n$ is a fixed positive integer, while $[n]=\{1,2,\dots,n\}$ and $\NN = \{0,1,2,3,\dots\}$.
This section first reviews the definitions of $\gl_n$- and $\sqgln$-crystals
from \cite{BumpSchilling,MT,Y}.
Section~\ref{svw-sect} recalls a useful model for normal $\sqgln$-crystals
and presents some related technical lemmas.
Section~\ref{rect-sect}, finally, 
contains a new result
showing how highest weight elements in normal $\sqgln$-crystals
can be obtained by applying a certain rectification operator.

\subsection{Abstract crystals}\label{crystal-sect}

We adopt the following setup throughout.
Suppose $\cB$ is a set with maps $\weight  :  \cB\to \ZZ^n$
and 
$e_i,f_i  :  \cB \to \cB \sqcup \{0\}$  for $i \in [n-1]$,
where $0 \notin \cB$.
Define $\varepsilon_i, \varphi_i  :  \cB \to \{0,1,2,\dots\}\sqcup \{\infty\}$ 
by 
\[
\label{var-eq}
\varepsilon_i(b) := \sup\left\{ k\geq 0 : (e_i)^k(b) \neq 0\right\}
\quand
\varphi_i(b) := \sup\left\{ k \geq 0: (f_i)^k(b) \neq 0\right\}.
\]

Write $\e_1,\e_2,\dots,\e_n$ for the standard basis of $\ZZ^n$.
The following definition is classical \cite{BumpSchilling}:
\begin{defn}
\label{crystal-def}
The set $\cB$ is a
 \definition{(seminormal) $\gl_n$-crystal}
if 
for all $i \in [n-1]$ and $b,c \in \cB$ we have
\ben
\item[(a)] both $\varepsilon_i(b)$ and $\varphi_i(b)$ are finite with $\varphi_i(b) - \varepsilon_i(b) = \weight(b)_i - \weight(b)_{i+1}$, and 

\item[(b)]  $e_i(b) = c $ if and only if $ f_i(c) = b$, 
in which case $ \weight(c) = \weight(b) + \e_{i} -\e_{i+1}$.
\een
\end{defn}

When these conditions hold, we refer to $\weight$ as the \definition{weight map} of $\cB$ and to $e_i$ and $f_i$ as  \definition{raising and lowering crystal operators}.
The \definition{crystal graph} of $\cB$ is the directed graph with vertex set $\cB$ that has a labeled edge $b\xrightarrow{i}c$ whenever $f_i(b) = c\neq 0$.
If this graph is weakly connected then we say that $\cB$ is \definition{connected}.
The \definition{highest weight elements} of $\cB$ are the source vertices in the crystal graph satisfying $e_i(b)=0$ for all $i \in [n-1]$.

Two crystals are \definition{isomorphic} if there is a weight-preserving graph 
isomorphism between their crystal graphs.
A \definition{full subcrystal} of $\cB$ is the set of vertices in some union of weakly connected components of the crystal graph of $\cB$. Such a set becomes a crystal by restricting the weight map and crystal operators. A \definition{connected component} of $\cB$ is a connected full subcrystal.

\begin{defn}\label{trivial-ex}
The \definition{trivial $\gl_n$-crystal} $\One_n=\{\emptyset\}$ has a single element of weight $0 \in \ZZ^n$, with 
all crystal operators acting as $e_i(\emptyset) = f_i(\emptyset) =0$.
The \definition{standard $\gl_n$-crystal} $\BB_n = \left\{ \boxed{i}: i \in [n]\right\}$ has the crystal graph shown in Figure~\ref{gl-fig}
along with the weight map $\weight(\boxed{i}):=\e_i \in \ZZ^n$.
\end{defn}

The following variant of Definition~\ref{crystal-def} was introduced in \cite{MT}:

\begin{defn}\label{sqgln-def}
The set $\cB$ is a \definition{$\sqgln$-crystal}
if for all $i \in[n-1]$
and $b,c \in \cB$ we have
\ben 
\item[(a)] 
both $\varepsilon_i(b)$ and $\varphi_i(b)$ are finite with 
$ \frac{\varphi_i(b)  - \varepsilon_i(b)}{2} = \weight(b)_i-\weight(b)_{i+1} \in \ZZ,$ and 
\item[(b)]   $e_i(b) = c$ if and only if $b = f_i(c)$,
in which case
$ \weight(c) - \weight(b) = \begin{cases} \e_i &\text{if $\varepsilon_i(b)$ is even} \\ -\e_{i+1}&\text{if $\varepsilon_i(b)$ is odd}.
\end{cases}
$
\een
\end{defn}

Crystal graphs, isomorphisms, highest weight elements, full subcrystals, and connected components for these objects are defined in the same way as for $\gl_n$-crystals.
Notice that the trivial crystal $\One_n$ from Definition~\ref{trivial-ex} is the only connected $\gl_n$-crystal (up to isomorphism) that is also
a $\sqgln$-crystal.

\begin{defn}\label{SS-ex}
The \definition{standard $\sqgln$-crystal} $\SS_n$ consists of all non-empty subsets $S\subseteq [n]$,
with weight map $\weight : \SS_n \to \NN^n$  given by  $\weight(S) = \sum_{i \in S} \e_i$
and with crystal operators defined by
\[ \ba e_i(S) &:= \begin{cases}
S\sqcup \{i\}&\text{if }S\cap \{i,i+1\} = \{i+1\}, \\
S\setminus \{i+1\}&\text{if }S\cap \{i,i+1\} = \{i,i+1\}, \\
0&\text{otherwise},
\end{cases}
\\[-10pt]\\
f_i(S) &:= \begin{cases}
S\sqcup \{i+1\}&\text{if }S\cap \{i,i+1\} = \{i\}, \\
S\setminus \{i\}&\text{if }S\cap \{i,i+1\} = \{i,i+1\}, \\
0&\text{otherwise}.
\end{cases}
\ea
\]
The crystal graph of $\SS_3$ is shown in Figure~\ref{SS3-fig}.
\end{defn}

Any $\sqgln$-crystal $\cB$ can be converted to a $\gl_n$-crystal 
through the following ``squaring'' operation from \cite[\S4.1]{MT}.
Define $\cB^{(2)}$ to have the same elements and weight map as $\cB$,  but 
with crystal operators by $e_i^2$ and $f_i^2$, which are evaluated by setting $e_i(0)=f_i(0)=0$.
The resulting object is always a (seminormal) $\gl_n$-crystal \cite[Prop.~4.4]{MT}.

The \definition{character} of a finite crystal $\cB$ is the polynomial 
$
\textstyle \ch(\cB) = \sum_{b \in \cB} x^{\weight(b)} \in \NN[x_1,x_2,\dots,x_n].
$
If $\cB$ is a seminormal $\gl_n$-crystal then this polynomial is symmetric  \cite[\S2.6]{BumpSchilling}.
As $\ch(\cB)=\ch(\cB^{(2)})$, the same is true when $\cB$ is a finite $\sqgln$-crystal
 \cite[Prop.~4.4]{MT}.

\begin{exa}
If $\lambda=(1)$ then we have 
\[
\ba \ch(\BB_n) &= x_1+x_2+\dots + x_n &&= s_{\lambda}(x_1,x_2,\dots,x_n),
\\
\textstyle\ch(\SS_n) &= \textstyle\sum_{k=1}^n \sum_{1\leq i_1 < i_2<\dots <i_k \leq n} x_{i_1}x_{i_2}\cdots x_{i_n} &&= G_\lambda(x_1,x_2,\dots,x_n).
\ea\]
\end{exa}

 Suppose $\cB$ and $\cC$ are both $\gl_n$-crystals or both $\sqgln$-crystals.
Then the set of formal tensors 
\[\cB \otimes \cC := \{ b \otimes c : b\in \cB\text{ and }c\in \cC\}\]
 has a unique crystal structure (of respective type $\gl_n$ \cite[\S2.2]{MT} or $\sqgln$ \cite[\S4.4]{MT}) 
in which the weight map is 
$
\weight(b\otimes c) = \weight(b) + \weight(c)
$
and  the crystal operators act as 
\begin{equation}\label{otimes-eq}\ba 
e_i(b\otimes c) = \begin{cases}
b \otimes e_i(c) &\text{if }\varepsilon_i(b) \leq \varphi_i(c) \\
e_i(b) \otimes c &\text{if }\varepsilon_i(b) > \varphi_i(c)
\end{cases}
\quand
f_i(b\otimes c) = \begin{cases}
b \otimes f_i(c) &\text{if }\varepsilon_i(b) < \varphi_i(c)\\
f_i(b) \otimes c &\text{if }\varepsilon_i(b) \geq \varphi_i(c),
\end{cases}
\ea
\end{equation}
where  we interpret $b\otimes 0 = 0\otimes c = 0$.
This tensor product is associative, for either category of crystals, in the sense that the natural maps $\cB \otimes (\cC \otimes \cD) \to (\cB \otimes \cC) \otimes \cD$ are
  isomorphisms.

Seminormal $\gl_n$-crystals are an abstraction for the \definition{crystal bases} of quantum group representations
first studied in \cite{Kashiwara1990,Kashiwara1991,Lusztig1990a,Lusztig1990b}.
The crystal tensor product just introduced corresponds to the natural tensor product for such representations.
We are particularly interested in the following classes of crystals that can be constructed using the tensor product:

\begin{defn}\label{normal-def1} A $\gl_n$-crystal is \definition{normal} if each of its connected components is isomorphic to a connected component of $\BB_n^{\otimes m}= \BB_n\otimes \BB_n\otimes \cdots \otimes \BB_n$ for some $m\in \NN$, where we define $\BB_n^{\otimes 0} =\One_n$.
 \end{defn}

The following variant of the preceding definition was introduced in \cite{MT}.

 \begin{defn}
A $\sqgln$-crystal is \definition{normal} if each of its connected components is isomorphic to a connected component of $\SS_n^{\otimes m}= \SS_n\otimes\SS_n \otimes \cdots\otimes \SS_n$ for some $m\in \NN$, where we again define $\SS_n^{\otimes 0}=\One_n$.
\end{defn}

The category of normal $\sqgln$-crystals is automatically closed under tensor products and disjoint unions,
but it is less well behaved
than its classical counterpart.
Unlike the $\gl_n$-case, there are
connected normal $\sqgln$-crystals with multiple highest weight elements; see Figure~\ref{multiple-fig} for an example.
There are also non-isomorphic normal $\sqgln$-crystals with the same highest weights and characters \cite[\S4.4]{MT}.
By contrast, the isomorphism class of any finite normal $\gl_n$-crystal is uniquely determined by its character.

These observations are consistent with the fact 
the squaring operation $\cB \mapsto \cB^{(2)}$,
which turns a $\sqgln$-crystal into a seminormal $\gl_n$-crystal,
does not commute with the crystal tensor product and does not usually preserve normality.

\begin{figure}[h]
\centerline{
\begin{tikzpicture}[xscale=7.5, yscale=1.5,>=latex,baseline=(z.base)]
    \node at (0,0.0) (z) {};
      \node at (0, 0) (T1) {$\{3\}\otimes \{1\}\otimes\{2\}\otimes \{1\}$};
      \node at (0,-1) (T2) {$\{3\}\otimes \{1,2\}\otimes\{2\}\otimes \{1\}$};
      \node at (0,-2) (T3) {$\{3\}\otimes \{2\}\otimes\{2\}\otimes \{1\}$};
      \node at (0,-3) (T4) {$\{3\}\otimes \{2\}\otimes\{2,3\}\otimes \{1\}$};
      \node at (0,-4) (T5) {$\{3\}\otimes \{2\}\otimes\{3\}\otimes \{1\}$};
      \node at (-1, 1) (U0) {$\{2,3\}\otimes \{1\}\otimes\{2\}\otimes \{1\}$};
      \node at (-1, 0) (U1) {$\{2,3\}\otimes \{1\}\otimes\{2,3\}\otimes \{1\}$};
      \node at (-1,-1) (U2) {$\{3\}\otimes \{1\}\otimes\{2,3\}\otimes \{1\}$};
      \node at (-1,-2) (U3) {$\{3\}\otimes \{1,2\}\otimes\{2,3\}\otimes \{1\}$};
      \node at (-1,-3) (U4) {$\{3\}\otimes \{1,2\}\otimes\{3\}\otimes \{1\}$};
      \node at (-1,-4) (U5) {$\{3\}\otimes \{1,2\}\otimes\{3\}\otimes \{1,2\}$};
      \node at (-1, -5) (U6) {$\{3\}\otimes \{2\}\otimes\{3\}\otimes \{1,2\}$};
      \draw[->,thick,blue]  (T1) -- (T2) node[midway,right,scale=0.75] {$1$};
      \draw[->,thick,blue]  (T2) -- (T3) node[midway,right,scale=0.75] {$1$};
      \draw[->,thick,red]  (T3) -- (T4) node[midway,right,scale=0.75] {$2$};
      \draw[->,thick,red]  (T4) -- (T5) node[midway,right,scale=0.75] {$1$};
      \draw[->,thick,red]  (U0) -- (U1) node[midway,right,scale=0.75] {$2$};
      \draw[->,thick,red]  (U1) -- (U2) node[midway,right,scale=0.75] {$2$};
      \draw[->,thick,blue]  (U2) -- (U3) node[midway,right,scale=0.75] {$1$};
      \draw[->,thick,red]  (U3) -- (U4) node[midway,right,scale=0.75] {$2$};
      \draw[->,thick,blue]  (U4) -- (U5) node[midway,right,scale=0.75] {$1$};
      \draw[->,thick,blue]  (U5) -- (U6) node[midway,right,scale=0.75] {$1$};
      \draw[->,thick,blue]  (U3) -- (T4) node[midway,above,scale=0.75] {$1$};
      \draw[->,thick,red]  (T2) -- (U3) node[midway,above,scale=0.75] {$2$};
     \end{tikzpicture}
}
     \caption{A connected component of $\SS_3^{\otimes 4}$ with 12 elements and 2 highest weight elements. This is an example of a connected normal $\sqrt{\gl_3}$-crystal with multiple highest weight elements.
    The character of the crystal is
     $x_1 x_2 x_3^2 + x_1 x_2^2 x_3 + x_1^2 x_2 x_3 + 2(x_1 x_2^2 x_3^2 + x_1^2 x_2 x_3^2 + x_1^2 x_2^2 x_3) + 3x_1^2 x_2^2 x_3^2$,
     which is equal to $G_{(2,2,1)}(x_1,x_2,x_3)+G_{(2,1,1)}(x_1,x_2,x_3)$ as predicted by Theorem~\ref{ch-thm}.
     }\label{multiple-fig}
     \end{figure}

\subsection{Set-valued words}\label{svw-sect}

The first half of this section reviews a concrete model for normal $\sqgln$-crystals introduced in \cite{Y}.
Within this model, all crystal operators can be computed by an explicit pairing process generalizing 
the \definition{signature rules} described in \cite[\S2.4]{BumpSchilling}.
The second half of this section presents some technical lemmas that will be needed later.

\begin{defn}
Fix integers $m\geq0$ and $n\geq 1$, and 
let $\SVWordsss{n}{m}$ denote the set of $m$-tuples 
\[ S = (S_1,S_2,\dots,S_m)\quad\text{where each $S_i$ is an arbitrary subset of $[n]=\{1,2,\dots,n\}$.}\]
We refer to elements of $\SVWordsss{n}{m}$ as \definition{set-valued words}.
We view $\SVWordsss{n}{m}$ as a normal $\sqgln$-crystal isomorphic to $(\One_n\sqcup\SS_n)^{\otimes m}$ by identifying 
$S$ with the formal tensor $S_1\otimes S_2\otimes \cdots \otimes S_m$.
\end{defn}

Every connected normal $\sqgln$-crystal is isomorphic to a connected component of $\SVWordsss{n}{m}$ for some $m$.
The weight function for $\SVWordsss{n}{m}$ is
$ \textstyle\weight(S) = \sum_{i \in [m]} \sum_{j \in S_i} \e_i \in \NN^n.$
To describe the crystal operators on $\SVWordsss{n}{m}$ we must 
recall some additional definitions from \cite[\S4]{Y}.

Fix an element $S =(S_1,S_2,\dots,S_m)\in \SVWordsss{n}{m}$.
For each $i \in [n-1]$, 
 the \definition{$i$-word} of $S$ is the following word composed of ``$($'', ``$)$'', and ``$-$''
characters concatenated together.
Read through the sets $S_j$ in $S$ from left to right.
For each set containing 
$i$ but not  $i+1$,
we write the single character ``$)$''.
For each set   containing
$i+1$ but not  $i$,
we write the single character ``$($''.
Finally, for each set containing both $i$ and $i+1$,
we write the three characters ``$)-($''. 

\begin{exa}
The set-valued word
\[S = \left(\ \{2,3\}, \ \{2\},\ \{2,3\},\ \{3,4\},\ \{1\},\ \{2\},\ \{1\}\ \right)
  \in \SVWordsss{4}{7}
\]
has 1-word ``$((()()$'',
2-word ``$)-())-(()$'',
and  3-word ``$)))-($''.
\end{exa}

The $i$-word is divided into equivalence classes in the following way. 
Ignore
the ``$-$'' symbols and pair the parentheses ``$($'' with ``$)$'' in the usual way. Then we require
that if a left parenthesis ``$($'' is paired with a right parenthesis ``$)$'', then these two characters and
everything between them are in the same equivalence class.
We also require that for each ``$)-($'', these three characters are in the same equivalence class.
Each of the resulting equivalence classes is a contiguous sub-word.

\begin{exa}
If the $i$-word is
``$))-(())-())-(()-($''
then its distinct equivalence classes are
\[
{)} 
\quand
{)-(())-()} 
\quand
{)-(}
\quand
{()-(}.\]
\end{exa}

Each equivalence class is designated as a \definition{null form}, a \definition{right form}, a \definition{combined form}, or a \definition{left form}, according to the following conventions:
\begin{enumerate}
\item[(a)] A class is a \definition{null form}
if it has no unpaired ``$($'' or ``$)$'' characters.
For example: ``$(()-())-()$''. 

\item[(b)] A class is a \definition{left form}
if it has no unpaired ``$)$'' characters but
ends with an unpaired ``$($''. 
\newline
This class is either ``$($'' or ``$(u)-($''
for some word $u$.
For example: ``$(())-($''. 

\item[(c)] A  class is a \definition{right form} if it
has no unpaired ``$($'' characters but
starts with an unpaired ``$)$''.
This class is either ``$)$'' or ``$)-(u)$''
for some word $u$.
For example: ``$)-()-()$''. 

\item[(d)] A class is a \definition{combined form} if it
starts with an unpaired ``$)$''   and
ends with an unpaired ``$($''.
This class is either ``$)-($'' or ``$)-(u)-($''
for some word $u$.
For example: ``$)-()-(())-($''. 
\end{enumerate}
There is always at most one combined form in an $i$-word. If we ignore the null forms then the $i$-word read from left to right consists of a sequence of zero or more right forms, followed by at most one combined form, followed by zero or more left forms.

\begin{prop}[\cite{Y}]
\label{sv-e-def}
If $i \in [n-1]$ then $e_i(S) \in\SVWordsss{n}{m}\sqcup \{0\}$ is given as follows:
\begin{itemize}
\item[(a)] If the $i$-word of $S$ has a combined form, 
then we find the entry in $S$ that corresponds to 
``$)-($'' at the end of this combined form.
We remove $i+1$ from this entry to obtain $e_i(S)$.
\item[(b)]  If the $i$-word of $S$ has no combined or left forms,
then we set $e_i(S) = 0$.
\item[(c)] Otherwise, find the first left form
in the $i$-word of $S$, and then 
find the entry in $S$ that corresponds to
``$($'' at the start of this left form.
Add $i$ to this entry to obtain $e_i(S)$.
\end{itemize}
\end{prop}

For completeness, we also recall the formula for the lowering operators on $\SVWordsss{n}{m}$.

\begin{prop}[\cite{Y}]
If $i \in [n-1]$ then $f_i(S) \in\SVWordsss{n}{m}\sqcup \{0\}$ is given as follows:
\begin{itemize}
\item[(a)] If the $i$-word of $S$ has a combined form, 
then we find the entry in $S$ that corresponds to 
``$)-($'' at the beginning of this combined form.
We remove $i$ from this entry to obtain~$f_i(S)$.
\item[(b)] 
If the $i$-word of $S$ has no combined or right forms,  
then we set $f_i(S) = 0$.
\item[(c)] Otherwise, find the last right form
in the $i$-word of $S$, and then
find the entry in $S$ that corresponds to
``$)$'' at the end of this right form.
Add $i+1$ to this entry to obtain $f_i(S)$.
\end{itemize}
\end{prop}

\begin{exa}
The following shows how $f_2$ is applied successively to an element of $\SVWordsss{4}{7}$:
\[
\begin{aligned}
(\ \{2,3\},\ 
\{2\},\ 
\{2\},\ 
\{3,4\},\  
\{1\},\ 
\{2\},\ 
\{1\}\ ) 
&\xrightarrow{\ f_2\ }
(\ \{2,3\},\ 
\{2\},\ 
\{2,3\},\ 
\{3,4\},\ 
\{1\},\ 
\{2\},\ 
\{1\}\ )
\\
&\xrightarrow{\ f_2\ }
(\ \{2,3\},\ 
\{2\},\ 
\{3\},\ 
\{3,4\},\ 
\{1\},\ 
\{2\},\ 
\{1\}\ )
\\
&\xrightarrow{\ f_2\ }
(\ \{2,3\},\ 
\{2,3\},\ 
\{3\},\ 
\{3,4\},\ 
\{1\},\ 
\{2\},\ 
\{1\}\ )
\\
&\xrightarrow{\ f_2\ }
(\ \{3\},\ 
\{2,3\},\ 
\{3\},\ 
\{3,4\},\ 
\{1\},\ 
\{2\},\ 
\{1\}\ ) \xrightarrow{\ f_2\ }0.
\end{aligned}
\]
\end{exa}

The formulas for $e_i(S)$ and $f_i(S)$ in these propositions first appeared in \cite{Y} as definitions.
To see why they give the same result as \eqref{otimes-eq} inductively applied to $(\One_n\sqcup\BB_n)^{\otimes m}$, see \cite[Thm.~4.28]{MT}.

The following useful observation is implicit in \cite{Y} and straightforward to derive from the previous two propositions.

\begin{cor}
Let $i \in [n-1]$ and $S \in \SVWordsss{n}{m}$.
Then $\varepsilon_i(S) = 2L+C$ and $ \varphi_i(S) = 2R+C$
where $L\in\NN$, $R\in\NN$, and $C\in\{0,1\}$ count the left, right, and combined forms in the $i$-word of $S$.
\end{cor}



We can easily identify the highest weight elements of $\SVWordsss{n}{m}$ 
in terms of certain  tableaux.
For this purpose, a \definition{tableau} means an
assignment of integers (or finite, non-empty sets of integers) to a finite set of boxes $(i,j) \in \PP\times\PP$.
As in Section~\ref{sym-intro}, we write $(i,j)\in T$ to indicate that $(i,j)$ is a box in a tableau $T$ and we write $T_{ij}$ to denote the value in this box.

\begin{defn}
An \definition{increasing tableau} is a tableau whose boxes make up the Young diagram of a partition and whose 
rows and columns are strictly increasing (in the direction of increasing indices).
Let $\IncT{n}{m}$ be the set of 
increasing tableaux filled by numbers from
$[m]$ with at most $n$ rows. 
\end{defn}

For each $S= (S_1,S_2, \dots, S_m) \in \SVWordsss{n}{m}$ define $\Tab(S)$ to be the tableau given by the following left-justified collection of boxes:
for each $i \in [n]$, the
 $i$th row of $\Tab(s)$ consists of the numbers $m+1-j$ listed in increasing order as $j$ ranges over the indices in $[m]$ with $i \in S_j$.
In other words, box $(i,c)$ of $\Tab(S)$
contains $j$ 
if the $c\textsuperscript{th}$
rightmost $i$ in $S$ is in $S_{m + 1 - j}$.
The shape of $\Tab(S)$ might not be the Young diagram of a partition, and its columns might not be increasing. However, $\Tab(S)$ has all entries in $[m]$ and at most $n$ rows, which are all strictly increasing.

\begin{exa} If  
$S=(\{1,3\},\{2\},\{1,2\},\{2\})$ then $m=4$ and 
$
\ytableausetup{aligntableaux=center}
\Tab(S) = \ytab{
2&4\\
1&2&3\\
4
}$.
\end{exa}

In the reverse direction,
given a tableau $T$ with all entries in $[m]$ and at most $n$ rows, all strictly increasing,
define $\gamma(T) \in \SVWordsss{n}{m}$ to be the set-valued word $(S_1,S_2,\dots,S_m)$
that has $i \in S_j$ if and only if  $m + 1 - j$ appears in row $i$ of $T$.
Clearly $\Tab(\gamma(T))=T$ and $\gamma(\Tab(S)) = S$.

\begin{prop}\label{tab-highest}
A set-valued word $S \in \SVWordsss{n}{m}$ is a highest weight element if and only if it holds that $\Tab(S) \in \IncT{n}{m}$.
Consequently, $\Tab$
defines a bijection
$\HW(\SVWordsss{n}{m})\to \IncT{n}{m}$.
\end{prop}

\begin{proof}
As explained in \cite[Lem.~4.11]{MT}, 
$S$ is a highest weight element 
if and only if $w(S)$ is a reverse lattice word.
This condition means that if $i+1 \in S_{j}$,
then among $S_{j+1}, \dots, S_m$,
the number of $i$'s is weakly larger than the
number of $i+1$'s.
Equivalently,
if the $c^\textsuperscript{th}$ rightmost $i+1$ in $S$
appears in $S_j$, 
then the $c^\textsuperscript{th}$ rightmost $i$
must exist and appear in $S_{j'}$ with $j' > j$.
We may translate this statement into
a statement about $\Tab(S)$:
whenever $\Tab(S)_{i+1, c}$ exists, 
$\Tab(S)_{i, c}$ also exists and $\Tab(S)_{i,c} < \Tab(S)_{i+1,c}$.
This is equivalent to saying that $\Tab(S)$
is an increasing tableau.
\end{proof}

\begin{exa}
The tuples $R=(\{1,3\},\{1\},\{1,2\},\{1\})$
and $S=(\{1,3\},\{1\},\{2\},\{1\})$
are two highest-weight elements of $\SVWordsss{3}{4}$.
Their corresponding increasing tableaux
are 
$$
\ytableausetup{boxsize = .6cm,aligntableaux=center}
\Tab(R) = \begin{ytableau}
1 & 2 & 3 & 4 \\
2\\
4
\end{ytableau}
\qquand
\Tab(S)=\begin{ytableau}
1  & 3 & 4 \\
2 \\
4
\end{ytableau}.
$$
\end{exa}

For the rest of this section we fix a set-valued word $S = (S_1,S_2,\dots,S_m) \in \SVWordsss{2}{m}$, that is, with each $S_j \subseteq \{1,2\}$. 
Define $E_1(S) := e_1^{k}(S)$ where $k=\varepsilon_1(S)$.
 Equivalently, $E_1(S)$ is the result of applying $e_1$ to $S$ as many times as possible before reaching zero.
We state two technical lemmas concerning  $E_1(S)$ and the tableau $\Tab(S)$.

Consider the equivalence classes in the $1$-word of $S$.
We say that one of these classes \definition{starts in position $i$} and \definition{ends in position $j$} if 
$S_i$ contains the number
that contributes the parenthesis at the start of the class
while $S_j$ contains the number that contributes the parenthesis at the end.
Suppose that in the $1$-word of $S$, the unique combined form, if it exists, starts in position $i_0$ and ends in position $j_0$.
Assume there are exactly $t$ left forms, which start in positions $i_1<i_2<\dots<i_t$ and end in positions $j_1<j_2<\dots<j_t$.

Notice that $S$ must have 
$S_{i_0} = S_{j_0}= \{1,2\}$  when there is a combined form.
Likewise, for each $k=1,2,\dots,t$,
one has 
$S_{i_k} = \{2\}$ when $i_k=j_k$,
while 
$S_{i_k} = \{2\}$ and $S_{j_k} = \{1,2\}$ when $i_k<j_k$.
The following is immediate from Proposition~\ref{sv-e-def}.

\begin{lemma}\label{E1-lem}
    In the preceding setup, the set-valued word $E_1(S)$ is formed from $S$ by removing 
    $2$ from each of the sets $S_{j_0},S_{j_1},S_{j_2},\dots, S_{j_t}$
    and adding $1$ to each of the sets $S_{i_1},S_{i_2},\dots, S_{i_t}$.
\end{lemma}

When $S$ is a highest weight element, its $1$-word must have no left or combined forms. In this event, the increasing tableau $\Tab(S)$ has some properties which we note for later use:

\begin{lemma}\label{highest-lem}
Assume $S \in \SVWordsss{2}{m}$ is a highest weight element. Let $T = \Tab(S)$.
\begin{enumerate}

\item[(a)]
Suppose the $1$-word of $S$ ends with a null form that starts in position $i$.
Then the number of entries in the subsequence $(S_i,S_{i+1},\dots,S_m)$ 
containing $1$ is the same as the number of entries containing $2$.
Suppose this number is $r$. Then $T_{2,r} = m + 1 - i$.
Moreover, $r$ is the smallest positive number such that
$(1, r+1) \notin T$ or $T_{2,r} < T_{1, r+1}$.

\item[(b)]
Suppose the $1$-word of $S$ ends with a right form.
Assume $T $ has $q$ boxes in row one and $p$ boxes in row two. Then 
$p<q$ and 
$T_{2,k} \geq T_{1,k+1}$ for all $1\leq k \leq p$.
\end{enumerate}
\end{lemma}

\begin{proof}
Assume we are in the situation of part (a). Then the $1$-word of $(S_i,S_{i+1},\dots,S_m)$ is a null form
where all parentheses are paired.
Thus, the number of left parentheses matches the number of 
right parentheses,
so $(S_i,S_{i+1},\dots,S_m)$ has the same number $r$ of $1$s
and $2$s.

By the definition of $\Tab(S)$, the tableau $T$ has
$T_{1,1} = m+1-j$ and $T_{2,r} = m+1-i$. Because $S$ is a highest weight element, $T$ is increasing.
Using this property, it is straightforward to check that if there is any $1\leq k < r$ with $T_{2,k} < T_{1,k+1}$, and $k$ is minimal with this property, then the 
terminal null form in the $1$-word of $S$ begins in position $m+1-T_{2,k} > m+1-T_{2,r}=i$.
As this contradicts the definition of $i$, we must have $T_{2,k} \geq T_{1,k+1}$ for all $1\leq k < r$.

Finally, if $(1,r+1) \in T$
then we cannot have $T_{2,r} \geq T_{1,r+1}$,
as then 
we would have $1 \in S_t$ for each of the $r+1$ indices $t=m+1-T_{1,k}$ with $1\leq k \leq r+1$,
and these indices would
all satisfy $t \geq m+1 - T_{1,r+1} \geq m+1 - T_{2,r} = i$.
This contradicts the fact that 
exactly $r$ of the sets $S_i,S_{i+1},\dots, S_m$ contain $1$.

Now there are $r$ sets containing $1$ and $r$ sets containing $2$ among $S_i, \dots, S_m$.
Since the null form starts in position $i$,
we know that $S_i = \{2\}$,
which is the $r^\textsuperscript{th}$
rightmost $2$.
Thus, $T_{2,r} = m+1 - i$.
The $(r+1)^\textsuperscript{th}$
rightmost $1$, if it exists,
is on the left of $S_i$,
so $T_{1,r+1} >  m+1 - i$.
Finally, assume 
toward a contradiction that
$T_{2, k} < T_{1, k+1}$
for some $r <k \leq m$.
Let $i' = m + 1 - T_{2, k}$,
so that $i' > i$.
We have  $T_{1, k} < T_{2, k} = m+1 - i' < T_{1, k+1}$.
Thus, among $S_{i'}, \dots, S_m$,
there are $k$ sets containing $1$ and $k$ sets containing $2$.
Since $S$ is a highest weight element, 
the $2$s correspond to paired right parentheses.
Thus, the $1$-word of $S_{i'}, \dots, S_m$ is non-empty
and consists of only paired parentheses,
so the last null form in $S$ cannot start at 
index $i$. This is a
contradiction.

The preceding arguments prove part (a). We turn to part (b). Assume the $1$-word of $S$ ends with a right form. It is clear from the definition of $\Tab(S)$ that $q$ is the number of entries of $S$ containing $1$ while $p$ is the number of entries containing $2$. 
Since $S$ is a highest weight element, its $1$-word has no left or combined forms, so $q-p$ is equal to the number of right forms in the $1$-word. As there is at least one of these by hypothesis, we must have $p<q$.

Finally, again
using the fact that $T$ is increasing, one checks that if there were any minimal $k$ with $1\leq k \leq p$ and $T_{2,k} < T_{1,k+1}$, then the last equivalence class in the $1$-word of $S$ would be a null form (starting in position $m+1-T_{2,k}$
and ending in position $m+1-T_{1,1}$).
This contradicts our assumption that the $1$-word ends in a right form, so $T_{2,k} \geq T_{1,k+1}$ for all $1\leq k \leq p$.
\end{proof}

\subsection{Rectification }\label{rect-sect}

This section  studies an operator that sends each element in a normal $\sqgln$-crystal to a highest weight element in the same connected component.

Let $\cB$ be a $\sqgln$-crystal and fix $i \in [n-1]$.
Each element $b\in \cB$ belongs to a unique path consisting of $\xrightarrow{i}$ edges, which we refer to as the  \definition{$i$-string} of $b$.
Generalizing the definition of $E_1$ in the previous section, let $E_i : \cB \to \cB$ be the operator that sends $b$ to the beginning of its $i$-string, so that 
$
E_i(b) = e_i^{k}(b)
$ for $k:=\varepsilon_i(b) \in \NN$. 
Then define 
\[ 
\rect := (E_1E_2 \cdots E_{n-1}) \cdots (E_1 E_2 E_3)\ (E_1 E_2) (E_1).
\]
This gives a map $\rect:\cB \to \cB$
with the property that $b$ and $\rect(b)$ always belong to the same connected component.
The goal of this section is to prove the following much less obvious result. 

\begin{thm}\label{rect-thm} If $\cB$ is a normal $\sqgln$-crystal then $\rect(b)$ is highest weight for all $b \in \cB$.
\end{thm}

An analogous result holds for normal $\gl_n$-crystals by \cite[Thm.~13.5]{BumpSchilling}.

\begin{exa}
If $S \in \SS_n$ then $E_i(S)$ is formed by changing all copies of $i+1$ in $S$ to $i$. Thus $E_i(S) = S$ if $S\cap\{i,i+1\}=\varnothing$, and otherwise $E_i(S) = (S\setminus\{i,i+1\})\sqcup \{i\}$.  It follows that 
\[
E_1E_2\cdots E_{i-1}(S)=\begin{cases}
S& \text{if $S$ is disjoint from $[i]$} \\
(S\setminus [i])\sqcup\{1\} & \text{otherwise}.\end{cases}
\]
We conclude that 
$\rect(S) = \{1\}$, so Theorem~\ref{rect-thm} holds if $\cB=\SS_n$.
\end{exa}

For the rest of this section fix (not necessarily normal) $\sqgln$-crystals $\cB$ and $\cC$.

\begin{lem}\label{delta-lem}
 For $b \in \cB$ and $c\in\cC$
 let
 $\Delta(b,c) := \max\{0, \varepsilon_i(b) - \varphi_i(c)\}
$. Then
\[
\varepsilon_i(b\otimes c) = \varepsilon_i(c) + \Delta(b,c)\quand
E_i(b \otimes c) = e_i^{\Delta(b,c)}(b) \otimes E_i(c)
\]
\end{lem}

\begin{proof}
This follows directly from the definitions of the crystal operators for tensor products. 
\end{proof}

\begin{lem}\label{any-lem}
If $b \in \cB$ and $c = E_i(b)$ then 
\[\weight(c)_{i+1} = \weight(b)_{i+1} - \lceil \varepsilon_i(b)/2\rceil \quand \weight(c)_i - \weight(c)_{i+1} = \tfrac{\varphi_i(b) + \varepsilon_i(b)}{2}.\]
\end{lem}

\begin{proof}
Observe that $\varphi_i(c) = \varphi_i(b) + \varepsilon_i(b)$ and $\varepsilon_i(c) = 0$. Then use the axioms in Definition~\ref{sqgln-def}.
\end{proof}

The next lemma only holds for normal $\sqgln$-crystals.

\begin{lem}\label{nor-lem}
Suppose $\cB$ is normal. Choose
 $i \in [n-2]$ and $b,c \in \cB$ with $e_{i+1}(b)=c$. 
 \begin{enumerate}
     \item[(a)] If $\varepsilon_{i+1}(b)$ is odd then
     $\weight(c)_i - \weight(c)_{i+1} = \weight(b)_i - \weight(b)_{i+1}$ and $\varepsilon_i(c) = \varepsilon_i(b)$.
     
     \item[(b)] If $\varepsilon_{i+1}(b)$ is even then
     $\weight(c)_i - \weight(c)_{i+1} = \weight(b)_i - \weight(b)_{i+1} - 1$
     and
     $\varepsilon_i(c) - \varepsilon_i(b)\in \{0,1,2\}$.
 \end{enumerate}
\end{lem}

\begin{proof}
Notice that if $\varepsilon_{i+1}(b)$ is odd then applying $e_{i+1}$ to $b$ does not change components $i$ or $i+1$ of the weight by Definition~\ref{sqgln-def}(b), while if $\varepsilon_{i+1}(b)$ is even then 
applying $e_{i+1}$ to $b$ subtracts one from component $i+1$ of the weight while leaving component $i$ unchanged. This implies our claims about how $\weight(c)_i -\weight(c)_{i+1}$ compares to $\weight(b)_i-\weight(b)_{i+1}$, and it just remains to show that 
\begin{equation}\label{jusn-eq}
\varepsilon_{i+1}(b)\text{ odd} \ \ \Rightarrow\ \ \varepsilon_i(c) = \varepsilon_i(b)
\qquand
\varepsilon_{i+1}(b)\text{ even} \ \ \Rightarrow\ \ \varepsilon_i(c) - \varepsilon_i(b) \in \{0,1,2\}.
\end{equation}
When $\cB=\One_n$ this property is obvious.
When $\cB=\SS_n$, 
we have $b\cap \{i\} = c\cap \{i\}$ and  either $b\cap \{i+1\} = c\cap \{i+1\}$ when $\varepsilon_{i+1}(b) \in \{0,1\}$ or $b\cap \{i+1\} \subsetneq c\cap \{i+1\}$ when $\varepsilon_{i+1}(b)=2$,
so the desired implications are again evident.

To complete the proof, it suffices by induction to 
assume that \eqref{jusn-eq} holds for $\cB$ and then check that it also holds for $\SS_n \otimes \cB$.
Following this strategy, we fix $i \in [n-2]$, $S \subseteq [n]$,  and $b \in \cB$ with $e_{i+1}(S\otimes b) \neq 0$.
If $e_{i+1}(S\otimes b) = S\otimes e_{i+1}(b)$ then
\[ 
\varepsilon_i(e_{i+1}(S\otimes b)) = \varepsilon_i(e_{i+1}(b)) +\max\{0,\varepsilon_i(S) - \varphi_i(e_{i+1}(b))\}.
\]
By hypothesis, we may assume that \eqref{jusn-eq} applies to $c=e_{i+1}(b)$.
This means that if $\varepsilon_i(b)$ is odd then $\varepsilon_i(e_{i+1}(b)) = \varepsilon_i(b)$, and also 
$\varphi_i(e_{i+1}(b)) = \varphi_i(b)$
by Definition~\ref{sqgln-def}(a), so 
\[
\varepsilon_i(e_{i+1}(b)) +\max\{0,\varepsilon_i(S) - \varphi_i(e_{i+1}(b))\}=
\varepsilon_i(b) +\max\{0,\varepsilon_i(S) - \varphi_i(b)\}.
\]
Thus $\varepsilon_i(e_{i+1}(S\otimes b))=\varepsilon_i(S\otimes b)$ as desired
by Lemma~\ref{delta-lem}.
If $\varepsilon_i(b)$ is even
then $\varepsilon_i(e_{i+1}(b)) = \varepsilon_i(b)+q$ and 
$\varphi_i(e_{i+1}(b)) = \varphi_i(b) + q-2$
for some $q \in \{0,1,2\}$ by Definition~\ref{sqgln-def}(a).
Then
\[ 
\varepsilon_i(e_{i+1}(S\otimes b)) = \varepsilon_i(b) + q
+
\max\{0,\varepsilon_i(S) - \varphi_i(b) +2-q\} \in \varepsilon_i(S\otimes b) + \{0,1,2\}
\]
as desired.

Now suppose $e_{i+1}(S\otimes b) =e_{i+1}(S) \otimes b$.
Then $T := e_{i+1}(S) \neq 0$
is obtained from $S$ by removing $i+2$ or adding $i+1$.
In the first case, 
one has $S\cap \{i+1,i+2\} = \{i+1, i+2\}$
and $T = S \setminus \{i+2\}$, and therefore
$T\cap \{i,i+1\} = S\cap \{i,i+1\}$
and $\varepsilon_i(T) = \varepsilon_i(S)$.
Suppose we are in the second case.
Then one has $S\cap \{i+1,i+2\} = \{i+2\}$
and $T = S \sqcup \{i+1\}$.
We consider whether $i \in S$:
\begin{itemize}
\item If $i \in S$,
then $T\cap \{i,i+1\} = \{i,i+1\}$.
We have $\varepsilon_i(S) = 0$ and $\varepsilon_i(T) = 1$.
\item Otherwise, $T\cap \{i,i+1\} = \{i+1\}$.
We have $\varepsilon_i(S) = 0$ and $\varepsilon_i(T) = 2$.
\end{itemize}
Then it follows from 
Lemma~\ref{delta-lem}
that 
\[\varepsilon_i(e_{i+1}(S\otimes b)) = \varepsilon_i(b) +\max\{\varepsilon_i(T) - \varphi_i(b)\} \in \varepsilon_i(S\otimes b) + \{0,1,2\},\] which is what we needed to show.
\end{proof}

Applying the previous lemma successively gives the following.

\begin{cor}
\label{C: diff in epsilon}
Suppose $\cB$ is normal with $i\in[n-2]$ and $b \in \cB$. 
If $\varepsilon_{i+1}(b)$ is odd 
then $\varepsilon_i(E_{i+1}(b)) - \varepsilon_i(b) \leq \varepsilon_{i+1}(b) -1$,
and if
$\varepsilon_{i+1}(b)$ is even 
then $\varepsilon_i(E_{i+1}(b)) - \varepsilon_i(b) \leq \varepsilon_{i+1}(b)$.
\end{cor}

To deduce Theorem~\ref{rect-thm} we require one other technical lemma.

\begin{lemma}\label{eE-lem}
Suppose $\cB$ is normal. Let $i \in [n-2]$ and $b \in \cB$. Then 
$ e_{i+1}( E_i E_{i+1} E_i(b)) =0.$
\end{lemma}

\begin{proof}
   This is easy to check  when $\cB =\One_n$ or $\SS_n$.
   Assume the desired property holds for $\cB$. Then it suffices to show that the property still holds for $\SS_n\otimes \cB$.

Fix $i \in [n-2]$, $S \subseteq [n]$,  and $b \in \cB$. Then 
\[ E_i E_{i+1} E_i(S\otimes b)
= e_i^p e_{i+1}^q e_i^r(S) \otimes E_i E_{i+1} E_i (b)\]
for some non-negative integers $p$, $q$, and $r$. By the inductive hypothesis,
$\varepsilon_{i+1}(E_i E_{i+1} E_i (b)) = 0$.
Together with Lemma~\ref{delta-lem}, we have 
\[
\varepsilon_{i+1} (E_i E_{i+1} E_i(S\otimes b)) = \max\{0, \varepsilon_{i+1}(e_i^p e_{i+1}^q e_i^r(S)) - \varphi_{i+1}(E_i E_{i+1} E_i(b)).\}
\]
Thus, we just need to show that 
\begin{align}
\label{EQ: Goal in eEEE}
\varepsilon_{i+1}(e_i^p e_{i+1}^q e_i^r(S)) \leq \varphi_{i+1}(E_i E_{i+1} E_i(b)).
\end{align}
The left hand side is at most two and nonzero if and only if $i+2 \in  e_i^pe_{i+1}^q e_i^r(S)$, or equivalently if $i+2 \in  e_{i+1}^q e_i^r(S)$. Meanwhile, the right hand side of the previous  equation is the even integer
\[
2(\weight(E_i E_{i+1} E_i(b))_{i+1} - \weight(E_i E_{i+1} E_i(b))_{i+2}),\]
as by hypothesis $\varepsilon_{i+1}(E_i E_{i+1} E_i(b))=0$.
In view of these facts,
we reduce~\eqref{EQ: Goal in eEEE} to:
\[
\weight(E_i E_{i+1} E_i(b))_{i+1} = \weight(E_i E_{i+1} E_i(b))_{i+2}
\quad\Rightarrow \quad i+2 \notin e_{i+1}^q e_i^r(S).\] Let $T = e_i^r(S)$ and $c = E_i(b)$. Then $q = \max\{0,\varepsilon_{i+1}(T) - \varphi_{i+1}(c)\}$
and we want to show that
\[
\weight(E_i E_{i+1}(c))_{i+1} - \weight(E_i E_{i+1}(c))_{i+2}=0
\quad \Rightarrow\quad i+2 \notin e_{i+1}^q(T).\]

To this end, we first use Lemma~\ref{any-lem} to obtain the identity
\[\begin{aligned}
\weight(E_i E_{i+1}(c))_{i+1} - \weight(E_i E_{i+1}(c))_{i+2} &=
\weight( E_i E_{i+1}(c))_{i+1} - \weight(E_{i+1}(c))_{i+2}
\\&=
\weight( E_{i+1}(c))_{i+1} - \weight(E_{i+1}(c))_{i+2} - \lceil \varepsilon_i( E_{i+1}(c))/2\rceil 
\\&=
\frac{\varphi_{i+1}(c) + \varepsilon_{i+1}(c)}{2} - \lceil \varepsilon_i( E_{i+1}(c))/2\rceil  
.\end{aligned}
\]
Set $\delta = 1$ when $\varepsilon_{i+1}(c)$ is odd
and $\delta = 0$ otherwise. Then
Corollary~\ref{C: diff in epsilon} implies
\[
\frac{\varphi_{i+1}(c) + \varepsilon_{i+1}(c)}{2} - \lceil \varepsilon_i( E_{i+1}(c))/2\rceil 
\geq
\frac{\varphi_{i+1}(c) + \varepsilon_{i+1}(c)}{2} - \left\lceil \frac{ \varepsilon_{i+1}(c)-\delta+ \varepsilon_i(c)}{2}\right\rceil.
\]
As $\varepsilon_i(c) = 0$,
and since 
$\varphi_{i+1}(c)$ and $\varepsilon_{i+1}(c)$ have the same parity, the last expression is equal to 
\[
\frac{\varphi_{i+1}(c) + \varepsilon_{i+1}(c)}{2} -  \frac{ \varepsilon_{i+1}(c)-\delta}{2} 
=
\frac{\varphi_{i+1}(c) +\delta}{2}
=
\left\lceil \frac{ \varphi_{i+1}(c)}{2}\right\rceil .
\]
Putting everything together, we conclude that
$$\weight(E_i E_{i+1}(c))_{i+1} - \weight(E_i E_{i+1}(c))_{i+2} \geq \left\lceil \frac{ \varphi_{i+1}(c)}{2}\right\rceil. $$ Thus
it is only possible to have 
$\weight(E_i E_{i+1}(c))_{i+1} - \weight(E_i E_{i+1}(c))_{i+2}=0$
when $\varphi_{i+1}(c) = 0$, and
in that case one has $q=e_{i+1}(T)$ so it is evident that  $i+2 \notin e_{i+1}^q(T) = E_{i+1}(T)$ as needed.
\end{proof}

We now give the proof of the main theorem of this section:

\begin{proof}[Proof of Theorem~\ref{rect-thm}]
Lemma~\ref{eE-lem} shows that if $i \in[n-1]$ then $E_{i}  E_{i-1} E_{i} E_{i-1} = E_{i-1} E_{i} E_{i-1}$ in any normal $\sqgln$-crystal.
If $|i-j|>1$ then $E_i$ and $E_j$ clearly commute.
Thus, for $1 < i < j$:
\begin{equation}
\label{EQ: E and rect}
\begin{split}
E_i \circ (E_1 E_2 \cdots E_{j-1}) \circ E_{i-1}
& = (E_1 \cdots E_{i-2}) \circ (E_i E_{i-1} E_i E_{i-1}) \circ (E_{i+1}\cdots E_{j-1})\\
& = (E_1 \cdots E_{i-2}) \circ (E_{i-1} E_i E_{i-1}) \circ (E_{i+1}\cdots E_{j-1})\\
& = (E_1 \cdots E_{i-2}) \circ (E_{i-1} E_i) \circ (E_{i+1}\cdots E_{j-1}) \circ E_{i-1}\\
&=(E_1 E_2 \cdots E_{j-1}) \circ E_{i-1}.   
\end{split}
\end{equation}
Next, let
$
\rect_j := (E_1E_2 \cdots E_{j-1}) \cdots (E_1 E_2 E_3)\ (E_1 E_2) (E_1).
$
We prove that $E_i \circ\rect_j = \rect_j$ for all $1 \leq i < j \leq n$ by induction
on $j$.
The base case when $j=2$ is obvious. The inductive step is
\[
\begin{aligned}
E_i \circ \rect_j &= 
E_i  \circ (E_1E_2\cdots E_{j-1}) \circ \rect_{j-1} 
\\&=
E_i  \circ (E_1E_2\cdots E_{j-1}) \circ E_{i-1}\circ \rect_{j-1}
\\&=
(E_1E_2\cdots E_{j-1}) \circ E_{i-1} \circ \rect_{j-1}
\\&=
(E_1E_2\cdots E_{j-1}) \circ \rect_{j-1} =\rect_j,
\end{aligned}
\]
which holds since
the second and last
equalities follow from our
inductive hypothesis
and the third equality follows
from~\eqref{EQ: E and rect}.
Thus $E_i \circ\rect_j = \rect_j$ for all $1\leq i < j \leq n$.
The theorem follows by taking $j=n$.
\end{proof}

\section{Grothendieck positivity}\label{gp-sect}

This section contains our proof of Theorem~\ref{ch-thm}, which shows that the characters of all finite normal $\sqgln$-crystals have characters that are sums of symmetric Grothendieck polynomials. 
Our proof relies on a
nontrivial result about the \definition{Hecke insertion algorithm} from \cite{BKSTY}, which is proved in
Section~\ref{hecke-sect}.
Section~\ref{Tab-sect} combines this with the lemmas in Section~\ref{prelim-sect} to derive our main results.

\subsection{Hecke insertion}\label{hecke-sect}

\definition{Hecke insertion} is an algorithm introduced in~\cite{BKSTY},
which gives a bijection between certain
pairs of words and  certain
pairs of tableaux.
We briefly review its definition in this section, and then present two technical lemmas about the algorithm.

Throughout, we use the following  procedure to insert a number into an increasing tableau:

\begin{defn}[\cite{BKSTY}]
\label{h-def}
Suppose $T$ is an increasing tableau and $x\in\PP$. To \definition{insert} $x$ into some column of $T$, we first locate the smallest entry $y$ in the column with $x<y$. Then:
\begin{enumerate}
\item[(a)] When no such entry $y$ exists, we append $x$ to the end of the column if this forms an increasing tableau, and otherwise we leave $T$ unchanged.
In the first case, we say that the insertion ends at the newly added box containing $x$.
In the second case, if
the last box of the column is in row $i$, and the last box in row $i$ of $T$ is in column $j$,
then
the insertion ends at $(i,j)$.

\item[(b)] When such an entry $y$ does exist,  we replace $y$ by $x$ if this forms an increasing tableau, and otherwise we leave $T$ unchanged.
In either case, we say that the insertion of $x$ \definition{bumps} the number $y$ from the column.
\end{enumerate}
We write $\insss{x}{T}$
to denote the result of first inserting $x$ into the first column of $T$,
and then in each successive column inserting the number bumped during the previous step.
We say that the insertion of $x$ into $T$ ends at the box where the final step of this process ends.
\end{defn}

\begin{exa}
Various examples of Definition~\ref{h-def} appear in \cite[\S3.1]{BKSTY}.
For instance, if 
\[T = \ytab{1 & 2 & 3 \\ 3 & 4 & 5}
\quad\text{then}
\quad
\insss{1}{T} = 
\ytab{
1 & 2 & 3 & 5\\
3 & 4 & 5},
\quad
\insss{2}{T} = 
\ytab{
1 & 2 & 3 & 5\\
2 & 3 & 4},
\quand\insss{3}{T} = 
\ytab{
1 & 2 & 3\\
3 & 4 & 5}.\]
Inserting $1$ or $2$ into $T$ ends at position $(1,4)$, while inserting $3$ ends at box $(2,3)$.
\end{exa}

We use the term \definition{word} to mean a finite sequence of positive integers $I=I_1I_2\cdots I_r$.
A pair of words $(A, I)$ 
is a \definition{compatible sequence}
if $I$ is weakly increasing with the same length as $A$, and it holds that
$A_j > A_{j+1}$ whenever
 $I_j = I_{j+1}$.

The formal definition of \definition{Hecke insertion} from \cite{BKSTY} is now given as follows:

\begin{defn}[\cite{BKSTY}]
Suppose $(A,I)$ is a compatible sequence with $A$ and $I$  of length $r$.
Let $P_0, P_1, P_2,\dots,P_r$ be the sequence of increasing tableaux with $P_0=\emptyset$ and $P_k = \insss{A_k}{P_{k-1}}$.
Then define $Q_k$ to be the set-valued tableau of the same shape as $P_k$ that is formed from $Q_{k-1}$ by adding $I_k$ to the box where the insertion of $A_k$ into $P_{k-1}$ ends. 
Finally, we let
\[ \TT(A,I) = P_r
\quand 
\TQ(A,I) = Q_r.\]
We
define the pair $(P_r,Q_r)$
to be the image of $(A,I)$ under Hecke insertion.
\end{defn}

The \definition{weight} of a word $I=I_1I_2\cdots I_r$
is the integer tuple $\weight(I)$ that has
$x^{\weight(I)} = x_{I_1}x_{I_2}\cdots x_{I_r}$.
By construction, $\TT(A,I)$ is an increasing tableau,
and the set of numbers
appearing in $A$ (respectively, $I$) is the same as the set of numbers appearing in $\TT(A,I)$ (respectively, $\TQ(A,I)$).
More precisely, the set-valued tableau $\TQ(A,I)$ has weight
$\wt(\TQ(A,I))=\wt(I)$.

It turns out that $\TQ(A,I)$ is always semistandard. In fact, results in \cite{BKSTY} show that 
Hecke insertion is a bijection from compatible sequences 
$(A,I)$ to pairs $(P, Q)$,
where $P$ is an increasing tableau
and $Q$ is a semistandard set-valued tableau with the same shape as $P$.

\begin{exa}\label{AI-ex}
If $A = 424311$ and $I=112223$ then 
\[
\ytableausetup{boxsize = .6cm,aligntableaux=center}
\TT(A,I) = \begin{ytableau}
1 & 2 & 4 \\
3
\end{ytableau}
\quand
\TQ(A,I) = \begin{ytableau}
1 & 12 & 23 \\
2
\end{ytableau}.
\]
\end{exa}

When $T$ is a translation of an increasing tableau, we define $\insss{x}{T}$ by first moving $T$ so that its upper left box is $(1,1)$, then Hecke inserting $x$ into $T$, and then translating the result back so that the upper left box is back to where it started in $T$.

If $T$ and $U$ are tableaux with disjoint sets of boxes, then we define $T \sqcup U$ in the natural way.
The following lemma about Hecke insertion will be useful later.
We omit its proof, which
follows as a basic exercise from the definitions.

\begin{lemma}\label{ins-lem}
Let $T$ be an increasing tableau with exactly two rows.
Let $x<z$ be the entries in the first column of $T$ and let $U$ be the tableau formed from $T$ by omitting the first column. Then:
\begin{enumerate}
    \item[(a)] If $y$ is an integer with $y<x$ then $\insss{y}{T}$ is formed by moving the first row of $T$  to the right by one column and then adding $y$ to the vacant box in position $(1,1)$.
    \item[(b)] If $y$ is an integer with $x<y<z$ then $\insss{y}{T}=\ytab{ x\\ y}\sqcup \insss{z}{U}.$

    \item[(c)] If $T$ is not a rectangle and every box $(2,j) \in T$ has $T_{2,j}\geq T_{1,j+1}$,  then $\insss{x}{T}=T$. 
    
    \item[(d)] If there is a minimal $j$ with $T_{2,j}<T_{1,j+1}$ or with $(2,j) \in T$ and $(1,j+1)\notin T$, then $\insss{x}{T}$ is formed from  $T$ by moving all boxes $(1,k)$ with $j+1\leq k$ to the right by one column and then adding $T_{2,j}$ to the vacant box in position $(1,j+1)$.
\end{enumerate}
\end{lemma}

If $B$ is any collection of distinct positive integers (that is, a single number, a set, a strictly decreasing sequence, a one-row increasing tableau, and so forth) then we write $\insss{B}{T}$ to denote the result of successively Hecke inserting the numbers in $B$ into $T$ in decreasing order.
The following more difficult lemma 
will be needed in the next section.

\def\Plast{P^{\mathsf{last}}}
\def\Ptop{P^{\mathsf{rest}}}

\def\Qlast{Q^{\mathsf{last}}}
\def\Qtop{Q^{\mathsf{rest}}}

\def\Clast{X}

\def\Rtop{R^{\mathsf{top}}}
\def\Rbot{R^{\mathsf{bot}}}

\def\Stop{S^{\mathsf{top}}}
\def\Sbot{S^{\mathsf{bot}}}

\def\iequiv{\overset{i}\equiv}

\begin{lemma}\label{hecke-lem}
Let $P$ be a non-empty increasing tableau.
Write $\Plast$ for the last row of $P$
and let $\Ptop$ be the other rows. 
Suppose $B=\{b_1 >b_2> \cdots > b_t\}$ is a set of positive integers. Then $\insss{B}{P}$ is obtained by the following procedure. 
First 
compute $T = \insss{B}{\Plast}$, which is 
an increasing tableau
of at most two rows. 
Then insert the numbers in the first row of $T$
into $\Ptop$ in decreasing order to obtain an increasing tableau
$U$.
Finally, append the second row of $T$
under $U$.
\end{lemma}

\begin{exa}
If $P = \ytab{
1 & 2 & 3 & 5 & 6 \\
2 & 3 & 6 & 9 \\
3 & 4 & 7}$ 
and 
 $B = \{6,5,3,1\}$ then it holds that 
\[\insss{B}{P} = 
\ytab{
1 & 2 & 3 & 5 & 6 & 9 \\
2 & 3 & 6 & 7 & 9 \\
3 & 4 & 7 \\
5 & 6 
}
.\] 
In the notation of the lemma, we have $
\Plast = \ytab{3 & 4 & 7}
$
and
$
\Ptop = \ytab{
1 & 2 & 3 & 5 & 6 \\
2 & 3 & 6 & 9}
$,
along with
\[
T=\insss{B}{\Plast}=\ytab{1 & 3 & 4 & 6 & 7 \\ 5 & 6}
\quand
U= \insss{\{7,6,4,3,1\}}{\Ptop} = \ytab{
1 & 2 & 3 & 5 & 6 & 9 \\
2 & 3 & 6 & 7 & 9 \\
3 & 4 & 7
}.
\]
\end{exa}

\begin{proof}
We first decompose the $P$ into smaller pieces as in the following pictures:
\newcommand{\cnb}{\ynobottom}
\newcommand{\cnbt}{\ynobottom\ynotop}
\newcommand{\cnt}{\ynotop}
\[
P \ = \ \begin{young}[16pt][c] 
]=\cnb &  \cnb & \cnb &  \cnb & \cnb & \cnb & \cnb & =] \\ 
]=\cnbt & \Ptop\cnbt & \cnbt & \cnbt & \cnbt & \cnbt & =]\cnt  \\ 
]=\cnt & \cnt &  \cnt & \cnt & \cnt & =] \cnt  \\ 
]=  &   \Plast &    &    \\ 
\end{young}
\  = \ \begin{young}[16pt][c] 
]=\cnb & ]= \cnb & \cnb &  \cnb & \cnb & \cnb & \cnb & =] \\ 
]=C\cnbt & ]= \cnbt & Q\cnbt & \cnbt & \cnbt & \cnbt & =]\cnt  \\ 
]=\cnbt & ]= \cnbt &  \cnbt & \cnbt & \cnt & =] \cnt  \\ 
]=\cnt & ]= \cnt & \cnt  & \cnt  \\ 
\end{young}
\  = \ 
\begin{young}[16pt][c] 
]=\cnb & ]= \cnb & \cnb &  \cnb & \cnb & \cnb & \cnb & =] \\ 
]=\cnbt & ]= \cnbt & \Qtop\cnbt & \cnbt & \cnbt & \cnbt & =]\cnt  \\ 
]=\cnt & ]= \cnt &  \cnt & \cnt & \cnt & =] \cnt  \\ 
\Clast & ]=  &  \Qlast &   \\ 
\end{young}.
\]
More precisely, we consider the following subtableaux of $P$:
\begin{itemize}
    \item Let $C$ be the first column of $P$.
    
    \item Let $\Clast$ be the single box in the last row of $C$, and write $x\in\PP$ for its unique entry.

    \item Let $Q$ be the part of $P$ excluding $C$ so that $P = C \sqcup Q$.

    \item Let $\Qlast$ be the last row of $Q$ and let $\Qtop$ be the other rows, so that $\Plast = \Clast \sqcup \Qlast$.

\end{itemize}
While $P$, $\Ptop$, and $C$ are increasing tableaux,
the other objects just defined are translations of increasing tableaux. 

We  adopt the convention that if $R$ is a nonempty tableau occupying at most two consecutive rows, then $\Rtop$ denotes the (nonempty) first row of $R$ and $\Rbot$ denotes the (possibly empty) second row.
Let $B = \{b_1>b_2>\dots>b_t\}$ and define $R=\insss{B}{\Plast}$. This means that we can draw
\[\insss{B}{\Plast} \ =\ 
\begin{young}[16pt][c] 
]=  &  &  \Rtop &    &  & & =]   \\
]=  &  &  \Rbot & & =]
\end{young}
\] 
with $\Rbot = \varnothing$ if $\insss{B}{\Plast}$ has only one row.
The lemma is equivalent to the claim that
\begin{equation}\label{ins-claim}
\insss{B}{P} =  \insss{\Rtop}{\Ptop} \sqcup \Rbot
\end{equation}

We first reduce this identity to the case when $\min(B)\geq x$.
To this end, define
\[B^{<x} = \{ b \in B : b < x\}
\quand 
B^{\geq x} = \{ b \in B : b \geq x\}.
\]
Observe that 
$\Rbot = \insss{B^{\geq x}}{\Plast}^{\mathsf{bot}}$,
while
$\Rtop = B^{<x} \sqcup \insss{B^{\geq x}}{\Plast}^{\mathsf{top}}$ is obtained by first listing $B^{< x}$ in order then appending the first row of $\insss{B^{\geq x}}{\Plast}$. 
Also notice that  
\[
\insss{B}{P} =
\insss{B^{<x}}{\insss{B^{\geq x}}{P}}
\]
can computed from $\insss{B^{\geq x}}{P}$ by deleting the row directly after $\Plast$, then inserting $B^{<x}$ as usual, and then adding back the deleted row.
Combining these observations shows that if we knew \eqref{ins-claim} in the case when $B=B^{\geq x}$, then the same identity  would hold for all subsets $B$.

Thus, we may assume  $\min(B) \geq x$.
Suppose this inequality is strict and let $y := \min(B) >x$.
Let $Y$ be the one-box tableau containing $y$
directly below $\Clast$, as in the following picture:
\[\begin{young}[16pt][c] 
]=\cnb & \cnb & \cnb &  \cnb & \cnb & \cnb & \cnb & =] \\ 
]=\cnbt & \Ptop\cnbt & \cnbt & \cnbt & \cnbt & \cnbt & =]\cnt  \\
]=\cnt & \cnt &  \cnt & \cnt & \cnt & =] \cnt  \\ 
]= \Clast & ]=  &  \Qlast &   \\ 
]=] Y  
\end{young}
\  = \ \begin{young}[16pt][c] 
]=\cnb & ]= \cnb & \cnb &  \cnb & \cnb & \cnb & \cnb & =] \\ 
]=C\cnbt & ]= \cnbt & Q\cnbt & \cnbt & \cnbt & \cnbt & =]\cnt  \\ 
]=\cnbt & ]= \cnbt &  \cnbt & \cnbt & \cnt & =] \cnt  \\ 
]=\cnt & ]= \cnt & \cnt  & \cnt  \\ 
]=] Y
\end{young}.\]
Note that $Y$ is not contained in $P = \Ptop  \sqcup \Clast \sqcup \Qlast   =  C \sqcup  Q$.
Now, we have
\begin{equation}
\label{ins-eq0}
    \insss{B}{P} = C \sqcup \insss{B\setminus\{y\}}{Q} \sqcup Y.
\end{equation}
Let $S = \insss{B\setminus\{y\}}{\Qlast}$.
By induction on tableau size, we may assume that
\begin{equation}\label{ins-eq3}
\insss{B\setminus\{y\}}{Q} = \insss{\Stop}{\Qtop}\sqcup \Sbot.
\end{equation}
In this case it is clear that
$\Rtop =   \Clast \sqcup \Stop$
and
$
 \Rbot = \Sbot \sqcup Y$ along with
\[
\insss{\Rtop}{\Ptop}=\insss{(\Clast \sqcup \Stop)}{\Ptop} = C \sqcup \insss{\Stop}{\Qtop}.
\]
These identities are valid even when $\Qlast$ is empty, 
interpreting $\Sbot$ to be empty and $\Stop$ to be $B\setminus\{y\}$ arranged in increasing order.
By combining \eqref{ins-eq0} and \eqref{ins-eq3} we get the needed identity
\[
\insss{\Rtop}{\Ptop} \sqcup \Rbot
=
C \sqcup \insss{\Stop}{\Qtop} \sqcup \Sbot \sqcup Y=\insss{B}{P}.
\]

The lemma is evident if $B = \{x\}$ so, as our final case, assume that $x =\min(B)<\max(B)$.
Redefine $y := \min(B\setminus\{x\})$.
We claim that  $\insss{B}{P}$ and $\insss{B\setminus\{x\}}{P}$ have the same last row.
This is because the last two entries in the first column of $\insss{B\setminus\{x\}}{P}$ are $x$ followed by $y$.  
Thus, when we insert $x$ into the first column of $\insss{B\setminus\{x\}}{P}$ to form $\insss{B}{P}$, the first column is unchanged and its last entry $y$ is inserted into the second column. 
Then, this number either bumps an entry before the last row (as will all subsequent insertions), or the column is unchanged and its last entry is inserted into the third column, and so on. Later insertions of numbers less than $x$ will only bump entries before the last row.
It follows similarly that  
$\Rbot=\insss{B}{\Plast}^{\mathsf{bot}}$
is equal to 
$\insss{B\setminus\{x\}}{\Plast}^{\mathsf{bot}}$.

The previous case shows that the last row of $\insss{B\setminus\{x\}}{P}$ is equal to $\insss{B\setminus\{x\}}{\Plast}^{\mathsf{bot}}$.
Therefore, if the index of this row is $i$,
then we just need to show that 
\[
 \insss{B}{P}
\iequiv 
\insss{\Rtop}{\Ptop}
\]
where $\iequiv$ means equality outside row $i$.
A crucial observation needed to prove this is that
\begin{equation}
\label{2ins-eq0}
    \insss{B}{P} \iequiv  C \sqcup \insss{B\setminus\{x\}}{Q}.
\end{equation}
This identity follows from the observations in the previous paragraph. In more detail, from the previous case, we know that
$
\insss{B\setminus\{x\}}{P} =  C \sqcup \insss{B\setminus\{x,y\}}{Q} \sqcup Y.
$
When we insert $x$ to form $\insss{B}{P}$,
the first column $C$ does not change and $y$ is inserted into the second column. Subsequent insertions leave row $i$ unchanged,
but the sequence of inserted numbers is the same as the one that results from inserting $y$ into $\insss{B\setminus\{x,y\}}{Q}$,
and so \eqref{2ins-eq0} follows.

Once again, if $S = \insss{B\setminus\{x\}}{\Qlast}$,
then we may assume by induction 
that
\begin{equation}\label{2ins-eq3}
\insss{B\setminus\{x\}}{Q} = \insss{\Stop}{\Qtop}\sqcup \Sbot.
\end{equation}
As we have
$\Rtop =   \Clast \sqcup \Stop$
 it follows that
\[
\insss{\Rtop}{\Ptop}=\insss{(\Clast \sqcup \Stop)}{\Ptop} = C \sqcup \insss{\Stop}{\Qtop},
\]
and so combining \eqref{2ins-eq0} and \eqref{2ins-eq3} gives
$
\insss{\Rtop}{\Ptop}
=
C \sqcup \insss{\Stop}{\Qtop} \iequiv \insss{B}{P}
$.
\end{proof}

\subsection{Characters of square root crystals}\label{Tab-sect}

This section contains the proof of Theorems~\ref{ch-thm} and \ref{dhf-thm}. 
These will be derived
as consequences of
Theorem~\ref{P: Insertion = rect} and \ref{main-thm}, which
relate the highest weight elements of normal $\sqgln$-crystals to Hecke insertion.

Our starting point is the following slightly unintuitive method of 
converting a set-valued word to a compatible sequence.
 
\begin{defn}
Given $S \in \SVWordsss{n}{m}$
we construct words $A(S)$ and $I(S)$
in the following way.
For each $i \in [n]$, if $j_1<j_2<\dots<j_k$ are the indices $j \in [m]$ with $i \in S_j$ then we define
\[ 
A_i = (m+1-j_1)(m+1-j_2)\cdots(m+1-j_k)\quand I_i = ii\cdots i \text { ($k$ repetitions)}.\]
We then concatenate these sequences to form
 $A(S)=A_1A_2\cdots A_n$
and $I(S)=I_1I_2\cdots I_n$.
By construction $(A(S), I(S))$ is a compatible sequence
with $\wt(I(S)) = \weight(S)$.
\end{defn}

\begin{exa}\label{ASIS-ex}
If 
$S = (\{1,2\},\{2\},\{1\},\{2,3\})$
then $A(S) = 424311$ and $I(S) = 112223$.
\end{exa}

The following statement is a straightforward exercise:

\begin{prop}\label{compat-prop}
The map $S \mapsto (A(S),I(S))$ is a bijection between
$\SVWordsss{n}{m}$ and the set of compatible sequences
$(A, I)$ such that the entries in $A$ and $I$
are all in $[m]$ and $[n]$, respectively.
\end{prop}

We let $(\TT(S), \TQ(S))$ 
be the image of $(A(S), I(S))$ 
under the Hecke insertion; formally:
\begin{equation}
\TT(S) := \TT(A(S),I(S))
\quand
\TQ(S) := \TQ(A(S),I(S))
\end{equation}
for each $S \in \SVWordsss{n}{m}$.
If $S = (\{1,2\},\{2\},\{1\},\{2,3\})$ 
then we see from Examples~\ref{AI-ex} and \ref{ASIS-ex} that 
\[
\ytableausetup{boxsize = .6cm,aligntableaux=center}
\TT(S) = \begin{ytableau}
1 & 2 & 4 \\
3
\end{ytableau}
\quand
\TQ(S) = \begin{ytableau}
1 & 12 & 23 \\
2
\end{ytableau}.
\]
Recall from Proposition~\ref{tab-highest} that we associate
a tableau
$\Tab(S)$ to each set-valued word $S$.
The tableaux that arise in this way from
the highest-weight elements of $\SVWordsss{n}{m}$ are closely related to the Hecke insertion tableau $\TT(S)$ just defined.

\begin{thm}
\label{P: Insertion = rect}
If $S \in \SVWordsss{n}{m}$ then $\TT(S) =\Tab(\rect(S))$.
\end{thm}

To show this result,
we take a recursive approach
using the following technical lemma.

\begin{lemma}\label{tech-lem}
Let $S=(S_1,S_2,\dots,S_m) \in \SVWordsss{n}{m}$ and $i\in[n-1]$.
Suppose that 
\begin{itemize}
\item[] $a_1<a_2< \cdots< a_t$ are the indices $a \in [m]$ with $i \in S_a$, and 
\item[] $b_1 <b_2< \cdots < b_{t'}$ are the indices $b \in [m]$ with $i+1 \in S_b$.
\end{itemize}
Then  rows $i$ and $i+1$ of $\Tab(E_i(S))$ coincide with $\insss{B}{R}$ where 
\[B:=\{m + 1 - b_i : i \in [t']\}
\quand
R:=\begin{array}{|c|c|c|c|} \hline m+1 - a_t & \cdots & m+1 - a_2 & m+1 - a_1 \\ \hline\end{array}.\]
\end{lemma}

\begin{proof}
By the symmetry of all crystal operator definitions, it suffices to prove the lemma when $i=1$.
If $t=0$ then $E_1$ just  changes every $2$ appearing in an entry of $S$ to $1$,
and so
\[
\Tab(E_1(S)) = \begin{array}{|c|c|c|c|} \hline m+1 - b_{t'} & \cdots & m+1 - b_2 & m+1 - b_1 \\ \hline\end{array}= \insss{B}{\emptyset}=\insss{B}{R}.
\]
Alternatively, if $t'=0$ then $S=E_1(S)$ and \[\insss{B}{R}=\insss{\emptyset}{R} = R=\Tab(S)=\Tab(E_1(S)).\] From this point on we  assume that $t>0$ and $t'>0$.
Our argument is by induction on $t'$.
Let $\tilde B = B \setminus\{m+1-b_{t'}\}$ and form $\tilde R$ from $R$ by omitting box $(1,1)$.
The rest of the proof is divided into three cases according to the sign of the difference $a_t - b_{t'}$.

Case 1: Suppose 
$b_{t'} < a_t$.
Form $S'$ from $S$ by 
removing $1$ from $S_{a_t}$ and $2$ from $S_{b_{t'}}$.
The $2$ in $S_{b_{t'}}$ corresponds to a left parenthesis
followed by one or more right parentheses. 
Thus, the $1$-word of $S'$ is obtained from that of 
$S$ by removing this left parenthesis
and its paired right parenthesis.
 
In view of Lemma~\ref{E1-lem},\ we see that the set in position $a_t$ of $E_1(S')$ 
still does not contain $1$ while the set in position $b_{t'}$ still does not contain $2$.
Thus, $E_1(S)$ can be obtained from $E_1(S')$ by adding $1$ and $2$
to the corresponding positions.
Now,  it is clear
that 
\[
\Tab(S'') = 
\begin{array}{|c|}
    \hline 
    m+1-a_t \\  \hline
    m+1-b_{t'} \\ 
    \hline
\end{array} \sqcup (\text{the tableau $\Tab(E_1(S'))$ shifted to the right by one column}).
\]
On the other hand, because $m+1-b_{t'} > m+1-a_{t}$, we see from the definition of Hecke insertion that we can express
$\insss{B}{R}$ as the disjoint union
\[ \insss{B}{R}= 
\begin{array}{|c|}
    \hline 
    m+1-a_t \\  \hline
    m+1-b_{t'} \\ 
    \hline
\end{array} \sqcup \insss{\tilde B}{ \tilde R}.\] 
Since we may assume by induction that $\insss{\tilde B}{\tilde R}$ is exactly $\Tab(E_1(S'))$ shifted to right by one column, we conclude that 
$\Tab(E_1(S)) = \Tab(S'')=\insss{B}{R}$.

Case 2:
Suppose 
$j:=b_{t'} > a_t$.
Form $S'$ from $S$ by removing $2$ from $S_j$.
This $2$ corresponds to the last equivalence class in the $1$-word of $S$, which is a single-character  left form. Passing from $S$ to $S'$ has the effect of omitting this equivalence class and otherwise leaving the $1$-word unchanged.
Hence, by Lemma~\ref{E1-lem} we see that the set in position $j$ of 
$E_1(S')$
does not contain $1$ or $2$.
Form $S''$ from $E_1(S')$ by 
adding $1$ to this set.
Then it holds that  $E_1(S) = S''$.

Now, we may assume by induction that $\insss{\tilde B}{R} = \Tab(E_1(S'))$, and so we also have
\[
\label{ind-assume-eq}
\insss{B}{R} = \insss{m+1-j}{\insss{\tilde B}{R}} = \insss{m+1-j}{\Tab(E_1(S'))}.
\]
Because all entries of $S'$ 
containing $1$ or $2$ are in positions $1,2,\dots,j-1$,
the same is true of $E_1(S')$, and so
the increasing tableau
$\Tab(E_1(S'))$ has all entries greater than $m+1-j$.
Thus, Lemma~\ref{ins-lem} tells us that
$\insss{B}{R}$
is formed from 
$\Tab(E_1(S'))$ by shifting the first row to the right and placing $m+1-j$ in box $(1,1)$.
But this tableau is exactly $\Tab(S'')=\Tab(E_1(S))$ as desired.

Case 3: 
Finally suppose   $j := a_t = b_{t'}$.
Then $S_j$ contains both $1$ and $2$ and the numbers contribute a ``$)-($'' to the end of the last equivalence class in the $1$-word of $S$,
which is a combined or left form that starts in some position $i$ with $i\leq j$.
If this equivalence class is a combined form, then 
  $S_i$ will also contain both $1$ and $2$  
 and it is possible for us to have $i=j$.
 If the equivalence class is a left form then we must have $i<j$ and
$S_i$ will contain $2$ but not $1$.

Form $S'$ from $S$ by removing $2$ from $S_j$.
We may assume again by induction that 
$
\insss{B}{R} = \insss{m+1-j}{\Tab(E_1(S'))}.
$
The last entry of $E_1(S')$ containing $1$ is still in position $j$, so $m+1-j$ is in box $(1,1)$ of $\Tab(E_1(S'))$.
Thus, when inserting $m+1-j$ into $\Tab(E_1(S'))$ we are in case (c) or (d) of Lemma~\ref{ins-lem}.

Suppose the $1$-word of $S$ ends with a left form. Then $i<j$ and the $1$-word of $S'$ ends with a null form starting in position $i$. 
By Lemma~\ref{E1-lem},
the $1$-word of $S'$ also ends with a null form starting in position $i$.
In this case the set in position $i$ of $E_1(S')$
contains $2$ but not $1$,
so we can
form $S''$ from $E_1(S')$ by
adding $1$ to this set. 
Then it follows from 
Lemma~\ref{E1-lem} that $E_1(S) = S''$.

Suppose that exactly $r$ of the entries of $E_1(S')$ in positions $i,i+1,i+2,\dots,m$
contain $1$.
Since both $S''$ and $E_1(S')$ are highest weight elements,  
  $\Tab(S'')$ is formed from $\Tab(E_1(S'))$ by moving all boxes $(1,k)$ with $r+1\leq k$ to the right by one column and then adding $m+1-i$ to the vacant box in position $(1,r+1)$.
 Consulting Lemmas~\ref{highest-lem}(a) and \ref{ins-lem}(d), we see that this exactly describes $\insss{m+1-j}{\Tab(E_1(S'))}=\insss{B}{R} $,
 so $\insss{B}{R} = \Tab(S'') = \Tab(E_1(S))$ as needed.
 
When the $1$-word of $S$ ends in a combined form, the $1$-word of $S'$
 ends with a right form, and by Lemma~\ref{E1-lem} we have $S' =E_1(S')= E_1(S)$. To deduce that $\insss{B}{R} = \Tab(E_1(S))$ we just need to check that $\insss{m+1-j}{S'} = S'$, and this holds by Lemmas~\ref{highest-lem}(b) and \ref{ins-lem}(c).
\end{proof}

Combining Lemmas~\ref{hecke-lem} and \ref{tech-lem} leads to this composite statement:

\begin{lemma}\label{med-lem}
Suppose $S \in  \SVWordsss{n}{m}$.
Form $S' \in \SVWordsss{n-1}{m}$ by setting $S'_i = S_i \cap [n-1]$ for all $i \in [m]$.
Assume $S'$ is highest-weight  in $\SVWordsss{n-1}{m}$
and let $P = \Tab(S')$ be its increasing tableau.
If  $b_1 <b_2< \cdots <b_t$ 
are the indices $b \in [m]$ with $n \in S_b$,
then 
\[\Tab(E_1E_2 \cdots E_{n-1}(S))= \insss{B}{P}
\quad\text{for }
B := \{m+1-b_i : i \in [t]\}
.\] 
\end{lemma}

\begin{proof}
This result is trivial when $n=1$ so assume $n\geq 2$. We argue by induction on $n$.
    Let $\Plast$ be the last row of $P$ and form $\Ptop$ from $P$ by omitting $\Plast$. Define $\Rtop$ and $\Rbot$ to be the first two rows of $\insss{B}{\Plast}$, with $\Rbot=\emptyset$ if $\insss{B}{\Plast}$ has only one row.

    Lemma~\ref{tech-lem} tells us that $\Ptop \sqcup \Rtop \sqcup \Rbot = \Tab(E_{n-1}(S))$.
    Thus, if we define $W'$ and $W''$ to be the set-valued words obtained by intersecting each
    entry of $E_{n-1}(S)$ with $[n-1]$ and $[n-2]$, respectively,
    then $\Ptop = \Tab(W'')$
    and
    we may assume by induction that
    \[
    \Tab(E_1E_2\cdots E_{n-2}(W')) = \insss{\Rtop}{\Ptop}.
    \]
   Since it is clear that 
   \[
\Tab(E_1E_2\cdots E_{n-1}(S)) = \Tab(E_1E_2\cdots E_{n-2}(W')) \sqcup \Rbot,
   \]
   we conclude from Lemma~\ref{hecke-lem} that $\Tab(E_1E_2\cdots E_{n-1}(S)) = \insss{\Rtop}{\Ptop} \sqcup \Rbot= \insss{B}{P}$.
\end{proof}

We now use the preceding lemma to quickly prove Theorem~\ref{P: Insertion = rect}.

\begin{proof}[Proof of Theorem~\ref{P: Insertion = rect}]
This result is also trivial when $n=1$ so assume $n\geq 2$. We argue by induction on $n$.
Again let 
$
\rect_j := (E_1E_2 \cdots E_{j-1}) \cdots (E_1 E_2 E_3) (E_1 E_2) (E_1)
$ for $j \in [n]$.
    
    Define $S'$ and $B$ as in Lemma~\ref{med-lem}.
   Then $\rect_{n-1}(S')$ is obtained by intersecting each entry of 
     $\rect_{n-1}(S)$ with $[n-1]$, 
     and the entries of $\rect_{n-1}(S)$ containing $n$ occur in the same positions as in $S$.
    As $\rect_{n-1}(S')$
    is a highest weight element of $\SVWordsss{n-1}{m}$ by Theorem~\ref{rect-thm},
 Lemma~\ref{med-lem}
    implies that 
    $
    \Tab(\rect(S)) 
    = \Tab(E_1E_2\cdots E_{n-1}(\rect_{n-1}(S))) = \insss{B}{\Tab(\rect_{n-1}(S'))} .
    $
    
    We may assume by induction that
    $\Tab(\rect_{n-1}(S'))=\TT(S')$.
    The theorem now follows since the identity
    $\insss{B}{\TT(S')} = \TT(S)$ holds by definition.
\end{proof}

From these results  we deduce
the following main theorems.

\begin{thm}\label{main-thm}
Suppose
$\cB\subseteq \SVWordsss{n}{m}$ is a full subcrystal.
Then 
$S \mapsto (\TT(S), \TQ(S))$
is a bijection from $\cB$ to
the set of pairs of tableaux $(P,Q)$ with the same shape such that 
\begin{enumerate}
    \item[(a)] $P$ is increasing and equal to $\Tab(S)$ for some highest weight element $S \in \cB$, and
    \item[(b)] $Q$ is semistandard and set-valued with all entries contained in $[n]$.
\end{enumerate}
\end{thm}

\begin{proof}
Let $\cP = \{\Tab(S) : S\in \HW(\cB)\}$.
Every element of this set is in $\IncT{n}{m}$ by Proposition~\ref{tab-highest}, and $\Tab(S)\notin \cP$ for all $S \in \SVWordsss{n}{m}$ with $S\notin \HW(\cB)$.
Theorems~\ref{rect-thm} and \ref{P: Insertion = rect}
therefore tell us that
if $S \in \SVWordsss{n}{m}$ then $\TT(S) \in \cP$ precisely when $S \in \cB$. 
Given this observation, the theorem follows from the fact that 
$S \mapsto (\TT(S), \TQ(S))$ 
is already known---from \cite{BKSTY} via Proposition~\ref{compat-prop}---to be a bijection from $\SVWordsss{n}{m}$ to
the set of pairs of tableaux $(P,Q)$ of the same shape,
where $P\in \IncT{n}{m}$ and $Q$ is semistandard and set-valued with all entries  in $[n]$.
\end{proof}

Our main application of the previous result is to deduce Theorem~\ref{ch-thm} from the introduction.
Recall that $\HW(\cB)$ denotes the set of highest weight elements in a $\gl_n$- or $\sqgln$-crystal $\cB$.

\begin{proof}[Proof of Theorem~\ref{ch-thm}]
Assume $\cB$ is a finite normal $\sqgln$-crystal.
We wish to show that 
\[\textstyle \ch(\cB) = \sum_{b \in \HW(\cB)} G_{\weight(b)}(x_1,x_2,\dots,x_n).\]
  Theorem~\ref{main-thm} implies this result when $\cB$ is a full subcrystal of $\SVWordsss{n}{m}$,
since it holds that $\weight(S) = \weight(\TQ(S))$ for all $S \in \SVWordsss{n}{m}$
and since $\weight(S)$ is the partition shape of both $\TT(S)$ and $\TQ(S)$ when $S$ is highest weight.
As every connected normal $\sqgln$-crystal is isomorphic to a full subcrystal of $\SVWordsss{n}{m}$ for some $m$,
the theorem follows.
\end{proof}

Suppose $\cB$ is a $\sqgln$-crystal.
Recall from Section~\ref{crystal-sect} that there is a seminormal $\gl_n$-crystal  $\cB^{(2)}$ with the same elements and weight map.
Define 
$\textstyle
\minlevel(\cB) = \min\left\{ \sum_{i \in[n]} \weight(b)_i : b \in \cB\right\}$
and
\begin{equation}
\textstyle
    \cB^{(2:\delta)} = \left\{ b\in \cB : \sum_{i \in [n]} \weight(b)_i = \minlevel(\cB) + \delta\right\}
\quad\text{for each $\delta \in \NN$.}
\end{equation}
Each of these subsets is a full $\gl_n$-subcrystal of $\cB^{(2)}$ and so may be viewed as a seminormal $\gl_n$-crystal on its own.

\begin{cor}\label{2d-cor}
If $\cB$ is a finite normal $\sqgln$-crystal then the character of $\cB^{(2:\delta)}$ is Schur positive.
\end{cor}

\begin{proof}
Each homogeneous component of $G_\lambda(x_1,x_2,\dots,x_n)$ is Schur positive.
This follows from \cite[Eq.~(6.5)]{Buch2002} (see also \cite[Thm.~2.2]{Lenart} and \cite[Cor.~3.11]{MoPeSc} for other proofs) after adjusting for signs,
since what is denoted ``$G_\lambda(x)$'' in \cite{Buch2002} is $(-1)^{|\lambda|}G_\lambda(-x)$ in our notation.
The corollary follows from this observation and
Theorem~\ref{ch-thm}, since the character of $\cB^{(2:\delta)}$ is the homogeneous component of $\ch(\cB)$ of degree $\minlevel(\cB)+\delta$.
\end{proof}

\begin{exa}
If $\cB = \SVT_n(\lambda)$, then
applying the corollary to $\cB^{(2:1)}$ shows that the weight-generating function for the \definition{barely set-valued tableaux} \cite{RTY} of a given shape is Schur positive.
 \end{exa}

Using Theorem~\ref{P: Insertion = rect}, we can now also prove Theorem~\ref{dhf-thm} from the introduction.
Fix a permutation $w \in S_\infty$ and recall the definition of $\DHF_n(w)$ from Section~\ref{pi-sect}.

Choose $m \in \PP$ such that $w \in S_{m+1}$. Then the Hecke words of $w$ have all letters in $[m]$.
Given $a \in \DHF_n(w)$, define
$\svword(a) \in \SVWordsss{n}{m}$
to be the set-valued word \begin{equation}
    \svword(a):=(S_1,S_2,\dots,S_m) \quad\text{where $S_j = \{ i \in [n]: m+1-j \in a^i\}$.}
\end{equation}
For example,
if $m=3$ and $w=1432\in S_4$ then $\svword$ acts on $\DHF_2(w)$ from Example~\ref{pi-ex} as
\[
(2,32) \mapsto 
 (\{2\}, \{1,2\}, \emptyset),
 \quad
(32,3)\mapsto 
 (\{1,2\}, \{1\}, \emptyset), 
 \quand
(32,32)\mapsto 
    (\{1,2\}, \{1,2\}, \emptyset).
\]
Recall the definitions of $\concat(a)$ and $\Tab(a)$ for $a \in \DHF_n(w)$ from Section~\ref{pi-sect}.
After unpacking these notations, the following properties are evident:

 \begin{lemma}\label{rts-lem}
If $a \in \DHF_n(w)$  then
it holds that $\concat(a) = \revrow(\Tab(a))$
along with 
\[\Tab(a) = \Tab(\svword(a))
 \quand 
\TT(\concat(a))=\TT(\svword(a))
 .\]
 \end{lemma}

We may now present our delayed proof.

\begin{proof}[Proof of Theorem~\ref{dhf-thm}]
 We first argue that $\DHF_n(w)$  has a unique  $\sqgln$-crystal structure 
that makes the operation $\svword$ into a crystal embedding $\DHF_n(w) \to \SVWordsss{n}{m}$.
This structure will make $\DHF_n(w)$ into a normal $\sqgln$-crystal with the desired weight map.

 To this end, notice that $\svword$ is a bijection $\bigsqcup_{w \in S_{m+1}} \DHF_n(w) \to \SVWordsss{n}{m}$.
 Thus, 
    we only need to check that the crystal operators on $\SVWordsss{n}{m}$
    preserve the image of $\DHF_n(w)$  when they do not act as zero.
We show this  
using the fact that if $i \in \cH(w)$ then $\revrow(\TT(i)) \in \cH(w)$.
This property of Hecke insertion follows by \cite[Lem.~1]{BKSTY} after checking that
the row and column words of an increasing tableau are equivalent under the  relation \eqref{hecke-rel}; see \cite[Lem.~2.7]{Mar2019}.

 Choose $a\in \DHF_n(w)$. For some $v \in S_{m+1}$ there is $b \in \DHF_n(v)$
with $\rect(\svword(a)) = \svword(b)$.
As the sets $\DHF_n(w)$ are disjoint, it suffices by Theorem~\ref{rect-thm} 
    to show that  $v=w$, or equivalently that $\concat(b) \in \cH(w)$.
As $\Tab(\rect(S)) 
=  \TT(S)$ for any set-valued word $S \in \SVWordsss{n}{m}$ by Theorem~\ref{P: Insertion = rect}, 
we deduce from Lemma~\ref{rts-lem} that
\[ \Tab(b) = \Tab(\svword(b))=  \Tab(\rect(\svword(a))) 
=  \TT(\svword(a)) = \TT(\concat(a)).\]
Using Lemma~\ref{rts-lem} again along with the property mentioned in the previous paragraph, we get
\[  \concat(b) = \revrow(\Tab(b)) = \revrow(\TT(\concat(a))) \in \cH(w)\]
as needed.
This establishes the first part of Theorem~\ref{dhf-thm}. 

It remains to identify the highest weight elements in the $\sqgln$-crystal $\DHF_n(w)$. 
However, in view of Lemma~\ref{rts-lem}, the desired classification is clear from
Proposition~\ref{tab-highest}. 
\end{proof}

\begin{rem}
Suppose $\cB = \DHF_n(w)$ for some permutation $w \in S_\infty$. A \definition{reduced word} for $w$ is a Hecke word of minimal length.
The set $\cB^{(2:0)}$ 
is a seminormal $\gl_n$-crystal structure on the set of \definition{decreasing reduced factorizations} 
$\textstyle
 \left\{ a \in \DHF_n(w) : \concat(a)\text{ is reduced}\right\}.
$
This object was previously studied in \cite{MorseSchilling},
and turns out to be a normal $\gl_n$-crystal.
A nontrivial argument is required to establish this normality, as the squaring operation $\cB \mapsto \cB^{(2)}$ does not send normal $\sqgln$-crystals to normal $\gl_n$-crystals. 

The character of $\cB^{(2:0
)}$ in this instance is the \definition{Stanley symmetric function} of $w $ defined in \cite{Stan}.
Corollary~\ref{2d-cor}
implies that this function is Schur positive, which was first shown in \cite{EG}.
\end{rem}

\section{Square root Demazure crystals}

Suppose $\cB$ is a finite normal $\gl_n$-crystal.
For any $i \in [n-1]$ and $X \subseteq \cB$ define 
\begin{equation}\label{fkD-eq} \fkD_i (X) := \left\{ b \in \cB : e_i^k(b) \in X\text{ for some }k \in \NN\right\}.
\end{equation}
Given $w \in S_n$ with a reduced word $i_1i_2\cdots i_k$, we consider the \definition{Demazure crystal}
\begin{equation}\label{fkDfkD-eq}
\cB_w := \fkD_{i_1}\fkD_{i_2}\cdots \fkD_{i_k}(\HW(\cB)).
\end{equation}
This subset does not depend on the choice of reduced word for $w$ by \cite[Thm.~13.5]{BumpSchilling} (see also \cite{Kashiwara93}).
Its \definition{character} is defined to be $\ch(\cB_w) = \sum_{b \in \cB_w} x^{\weight(b)} \in \ZZ[x_1,x_2,\dots,x_n]$.

The Demazure crystal $\cB_w$ is usually not a full subcrystal of $\cB$, 
 and so is not necessarily a seminormal $\gl_n$-crystal in the sense of Definition~\ref{crystal-def}. However, if
we view this object as simply a directed graph
with vertices weighted by $\ZZ^n$ and edges labeled by $[n-1]$, then
$\cB_w$ 
is determined up to isomorphism by its character. Morever, this character 
is always a $\NN$-linear combination of \definition{key polynomials} 
as defined below.

In this final section we describe a square root analogue of the Demazure crystal $\cB_w$ and present a positivity conjecture about its character that would generalize Theorem~\ref{ch-thm}.

\subsection{Key and Lascoux polynomials}\label{sG-sect}

Key and Lascoux polynomials are non-symmetric generalizations
of Schur and Grothendieck polynomials, which are usually defined in terms of the following \definition{divided difference operators}.

Recall that  $x_1,x_2,\dots,x_n$ are commuting indeterminates.
The symmetric group $S_n$ acts on $\ZZ[x_1,x_2,\dots,x_n]$
by permuting variables.
Fix $i \in [n-1]$ and let $s_i = (i,i+1) \in S_n$, so that if $f \in \ZZ[x_1,x_2,\dots,x_n]$ then $s_if = f(\dots, x_{i+1},x_i,\dots)$.
Then set
\begin{equation}
\partial_i(f) := \tfrac{f - s_if}{x_i - x_{i+1}}
\quand
\partial^K_i(f) := \partial_i((1 + x_{i+1})f)  = \tfrac{(1 + x_{i+1})f - (1+ x_i)s_if}{x_i- x_{i+1}}.
\end{equation}
Both $\partial_i$ and $\partial^K_i$ are $\ZZ$-linear maps from $\ZZ[x_1,x_2,\dots,x_n]\to\ZZ[x_1,x_2,\dots,x_n]$.
We also define 
\begin{equation}
\pi_i (f) = \partial_i(x_i f)\quand  \pi^K_i( f)= \partial^K_i(x_i f).
\end{equation}
It is easy to check that these operators satisfy
\begin{equation}\label{dd-eq}
\partial_i\partial_i=0,\qquad \partial_i^K\partial_i^K=-1,\qquad \pi_i\pi_i=\pi_i, \qquand \pi^K_i\pi^K_i=\pi^K_i.\end{equation}

A \definition{weak composition}  
 is an $n$-tuple $\alpha=(\alpha_1,\alpha_2,\dots,\alpha_n)$ of non-negative integers.
The group $S_n$ acts on weak compositions on the right by permuting entries.
This means that if $i \in [n-1]$ then $\alpha s_i$ is the weak composition
obtained by swapping entries $\alpha_i$ and $\alpha_{i+1}$ of $\alpha$.
To each weak composition $\alpha$, there is an associated \definition{key polynomial} $\kappa_\alpha$
and \definition{Lascoux polynomial} $\fL_\alpha$. These polynomials are defined recursively by setting
$\kappa_\alpha = \fL_{\alpha} = x^\alpha$ when $\alpha$ is weakly decreasing,
and then requiring that 
\begin{equation}\label{kappa-eq}
\pi_i( \kappa_\alpha) = \begin{cases}
    \kappa_{\alpha s_i} & \text{if $\alpha_i>\alpha_{i+1}$}\\
    \kappa_\alpha &\text{if $\alpha_i \leq \alpha_{i+1}$}
    \end{cases}
\quand 
\pi^K_i( \fL_\alpha) = \begin{cases}
    \fL_{\alpha s_i} & \text{if $\alpha_i>\alpha_{i+1}$}\\
    \fL_\alpha &\text{if $\alpha_i \leq \alpha_{i+1}$.}
    \end{cases}
\end{equation}

Results in \cite{Mon16,ReinerShimozono} show that if $\lambda= (\lambda_1\geq \lambda_2 \geq \dots\geq \lambda_n\geq 0)$ is a partition with at most $n$ parts and $\rev(\lambda) := (\lambda_n,\dots,\lambda_2,\lambda_1)$ denotes its reversal, then 
\begin{equation}
s_\lambda(x_1,\dots,x_n) := \kappa_{\rev(\lambda)}
\quand
G_\lambda(x_1,\dots,x_n) := \fL_{\rev(\lambda)}.\end{equation}
In addition, if $\cB_w$ is the Demazure crystal defined in \eqref{fkDfkD-eq}, then one has 
\begin{equation}
\label{EQ: Demazure character}
\textstyle
    \ch(\cB_w) = \sum_{b \in \HW(\cB)}  \pi_{i_1}\pi_{i_2}\cdots \pi_{i_k} x^{\weight(b)}
\end{equation}
as a consequence of the \definition{Demazure character formula}  \cite[Thm.~13.7]{BumpSchilling}.
Notice that each summand is equal to $\kappa_\alpha$ for some weak composition $\alpha$ in the $S_n$-orbit of the partition $\weight(b) \in \ZZ^n$ by \eqref{kappa-eq}.

\subsection{Rectangular pipe dreams}

This section presents an equivalent way of constructing the $\sqgln$-crystal $\SVWordsss{n}{m}$ in terms of the following objects,
which are the special cases of 
(non-reduced) \definition{pipe dreams}
considered previously, for example, in \cite{bergeron-billey,KnutsonMiller2004,KnutsonMiller}.
 
\begin{defn}
A \definition{rectangular pipe dream (RPD)} of size $n$-by-$m$
is a rectangular array of $n$ rows and $m$ columns,
filled by two types of tiles drawn as $\bumptile$ (the ``bump'' tile) and $\ptile$ (the ``crossing'' tile).
We index the rows and columns using matrix coordinates,
so that row $1$ is the highest row 
and column $1$ is the leftmost column. 
Let $\RPD_{m,n}$ be the set of all such RPDs.
\end{defn}

Each element of $\RPD_{m, n}$ is associated with a permutation
in $S_{m + n}$.
To find this permutation, we trace the pipes that are drawn by the RPD's tiles in the following way.

Consider the pipes that enter from the left edge and bottom edge. 
The pipe entering from row $i$ of the left edge
has label $i$.
The pipe entering from column $j$ of the bottom edge
has label $n + j$.
When pipes with labels $a$ and $b$ enter
from the left and bottom edges of a $\ptile$,
we consider the pipe with label $\max(a,b)$
to exit from the top while the other pipe exits from the right. 
Finally, we obtain the permutation of the RPD
by reading the numbers from left to right on the top edge,
then from top to bottom on the right edge.

\begin{defn}\label{sigma-def}
Given $D \in \RPD_{m,n}$,
let $\sigma_D \in S_{m+n}$ be its associated permutation.
Then, 
for any permutation $w \in S_{m+n}$, 
let $\RPD_{m,n}(w) = \{D\in \RPD_{m,n} : \sigma_D = w\}$.
\end{defn}

The permutation $\sigma_D \in S_{m+n}$ associated to $D \in \RPD_{m,n}$
can also be defined in terms of the Demazure product from Section~\ref{pi-sect}.
List the positions of the $\ptile$ tiles in $D$ row by row from left to right, but starting with the bottom row. 
In other words, order these positions  
$(r_1,c_1),(r_2,c_2),\dots,(r_l,c_l)$ 
in the unique way such that $r_1 \geq r_2 \geq \dots \geq r_l$ and $c_j < c_{j+1}$ whenever $r_j = r_{j+1}$.
Then it holds that 
$\sigma_D=s_{i_1}\circ s_{i_2} \circ \cdots \circ s_{i_l}$
for the indices
$i_j := r_j + c_j - 1$.

In the terminology of 
\cite{KnutsonMiller2004,KnutsonMiller},
our family $\RPD_{m,n}(w)$
is the subset of \definition{(non-reduced) pipe dreams} for $w^{-1}$ with all crossings in the first $n$ rows and first $m$ columns.

\begin{exa}
\label{E: RPDs}
The following elements of $\RPD_{4,3}$ 
$$
\begin{tikzpicture}[x=1.5em,y=1.5em,thick,rounded corners, color = blue]
\draw[step=1,black, thin] (0,0) grid (4,3);
\node[color=black] at (-0.5,2.5) {$1$};
\draw(0, 2.5)--(.5,2.5)--(.5, 3);
\node[color=black] at (-0.5,1.5) {$2$};
\draw(0, 1.5)--(2.5,1.5)--(2.5, 2.5)--(3.5, 2.5)--(3.5, 3);
\node[color=black] at (-0.5,.5) {$3$};
\draw(0, .5)--(.5, .5)--(.5, 2.5)--(1.5,2.5)--(1.5, 3);
\node[color=black] at (.5,-0.5) {$4$};
\draw(.5, 0)--(.5, .5)--(4, .5);
\node[color=black] at (1.5,-0.5) {$5$};
\draw(1.5, 0)--(1.5, 2.5)--(2.5,2.5)--(2.5, 3);
\node[color=black] at (2.5,-0.5) {$6$};
\draw(2.5, 0)--(2.5, 1.5)--(4, 1.5);
\node[color=black] at (3.5,-0.5) {$7$};
\draw(3.5, 0)--(3.5, 2.5)--(4, 2.5);
\node[color=red] at (0.5,3.5) {$1$};
\node[color=red] at (3.5,3.5) {$2$};
\node[color=red] at (1.5,3.5) {$3$};
\node[color=red] at (4.5,0.5) {$4$};
\node[color=red] at (2.5,3.5) {$5$};
\node[color=red] at (4.5,1.5) {$6$};
\node[color=red] at (4.5,2.5) {$7$};
\end{tikzpicture}
\quad\quad
\begin{tikzpicture}[x=1.5em,y=1.5em,thick,rounded corners, color = blue]
\draw[step=1,black, thin] (0,0) grid (4,3);
\node[color=black] at (-0.5,2.5) {$1$};
\draw(0, 2.5)--(.5,2.5)--(.5, 3);
\node[color=black] at (-0.5,1.5) {$2$};
\draw(0, 1.5)--(2.5,1.5)--(2.5, 3);
\node[color=black] at (-0.5,.5) {$3$};
\draw(0, .5)--(.5, .5)--(.5, 2.5)--(1.5,2.5)--(1.5, 3);
\node[color=black] at (.5,-0.5) {$4$};
\draw(.5, 0)--(.5, .5)--(4, .5);
\node[color=black] at (1.5,-0.5) {$5$};
\draw(1.5, 0)--(1.5, 2.5)--(3.5,2.5)--(3.5, 3);
\node[color=black] at (2.5,-0.5) {$6$};
\draw(2.5, 0)--(2.5, 1.5)--(4, 1.5);
\node[color=black] at (3.5,-0.5) {$7$};
\draw(3.5, 0)--(3.5, 2.5)--(4, 2.5);
\node[color=red] at (0.5,3.5) {$1$};
\node[color=red] at (3.5,3.5) {$2$};
\node[color=red] at (1.5,3.5) {$3$};
\node[color=red] at (4.5,0.5) {$4$};
\node[color=red] at (2.5,3.5) {$5$};
\node[color=red] at (4.5,1.5) {$6$};
\node[color=red] at (4.5,2.5) {$7$};
\end{tikzpicture}
$$
are both
associated with the permutation 
\[1352764 =s_4 \circ s_5 \circ s_6 \circ s_2 \circ s_3 \circ s_5
=
s_4 \circ s_5 \circ s_6 \circ s_2 \circ s_3 \circ s_5 \circ s_3
\in S_7.\]
 Notice that in the left RPD there are no pipes that cross more than once, so the labels on the top and right edges can be read off by tracing the pipes in the intuitive way.
 However, in the right RPD there some pipes that cross multiple times. Specifically, the pipes labeled by $2$ on the left edge and by $5$ on the bottom edge cross twice. The $\ptile$ tile at either of these crossings could be replaced by  a $\bumptile$ without changing the associated permutation, since the top edge of a $\ptile$ is always labeled by the maximum of the labels on the left and bottom edges.
\end{exa}

The following operation is a straightforward bijection from $\SVWordsss{n}{m}$
to $\RPD_{m,n}$:
we send a set-valued word $(S_1, S_2,\cdots, S_m) \in \SVWordsss{n}{m}$
to the RPD whose tile at row $i$ column $j$
is $\bumptile$ if $i \in S_j$ and $\ptile$ if $i \notin S_j$. 
For instance, 
the RPD's in Example~\ref{E: RPDs} are respectively associated with 
\[(\{1,3\},\{1\},\{1,2\},\{1\})
\quand (\{1,3\},\{1\},\{2\},\{1\}).\]
We give $\RPD_{m,n}$ the unique (normal) $\sqgln$-crystal structure that makes this bijection an isomorphism.
The following result describes
the permutations associated to the RPDs in a given $i$-string for this crystal structure.
 
\begin{prop}\label{P: string and Permutation}
Let $D_0, D_1, \cdots, D_l$ be an $i$-string
in $\RPD_{m,n}$ with associated permutations $u_0, u_1,\cdots, u_l \in S_{m+n}$.
If row $i$ and row $i+1$ of $D_0$ only contain $\ptile$ tiles, then
we must have $l = 0$ as well as $u_0(m + i) = i$ and $u_0(m+i + 1) = i+1$.
Otherwise, the following holds:
\begin{itemize}
\item[(a)] The permutation $u_0$ has a  descent at $m+i$, in the sense that
$u_0(m + i) > u_0(m + i+1)$.
\item[(b)] The value $i$ appears after the value $i+1$ 
in $u_l$,
in the sense that $u_l^{-1}(i) > u_l^{-1}(i+1)$.
\item[(c)] It holds that $u_1 = \cdots = u_{l-1} \in \{u_0,u_0 s_{m+i}\} \cap \{u_l, s_i u_l\}$.
 
\end{itemize}
\end{prop}

Before giving the proof,
we explain this statement through an example.

\begin{exa}
\label{E: string}
Consider the following $1$-string from $\SVWordsss{2}{2}$:
$$
(\{1\},\{1\})\quad\xrightarrow{f_1}\quad
(\{1\},\{1,2\})\quad\xrightarrow{f_1}\quad
(\{1\},\{2\})\quad\xrightarrow{f_1}\quad
(\{1,2\},\{2\})\quad\xrightarrow{f_1}\quad
(\{2\},\{2\})
$$
which corresponds to the following $1$-string in $\RPD_{2,2}$:
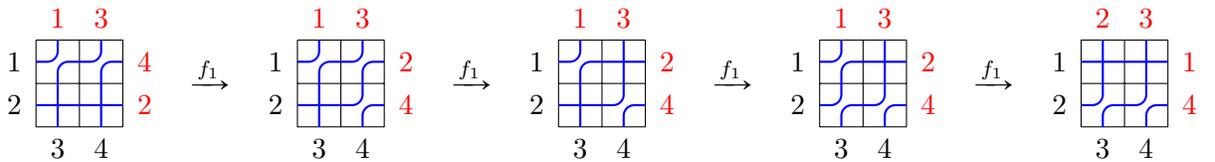
$$
\begin{tikzpicture}[x=1.5em,y=1.5em,thick,rounded corners, color = blue]
\draw[step=1,black, thin] (0,0) grid (2,2);
\node[color=black] at (-0.5,1.5) {$1$};
\draw(0, 1.5)--(.5,1.5)--(.5, 2);
\node[color=black] at (-0.5,.5) {$2$};
\draw(0, .5)--(2, .5);
\node[color=black] at (.5,-0.5) {$3$};
\draw(.5, 0)--(.5, 1.5)--(1.5, 1.5)--(1.5,2);
\node[color=black] at (1.5,-0.5) {$4$};
\draw(1.5, 0)--(1.5, 1.5)--(2, 1.5);
\node[color=red] at (.5,2.5) {$1$};
\node[color=red] at (1.5,2.5) {$3$};
\node[color=red] at (2.5,1.5) {$4$};
\node[color=red] at (2.5,0.5) {$2$};
\end{tikzpicture}
\quad\raisebox{1 cm}{$\xrightarrow{f_1}$}\quad
\begin{tikzpicture}[x=1.5em,y=1.5em,thick,rounded corners, color = blue]
\draw[step=1,black, thin] (0,0) grid (2,2);
\node[color=black] at (-0.5,1.5) {$1$};
\draw(0, 1.5)--(.5,1.5)--(.5, 2);
\node[color=black] at (-0.5,.5) {$2$};
\draw(0, .5)--(1.5, .5)--(1.5, 1.5)--(2, 1.5);
\node[color=black] at (.5,-0.5) {$3$};
\draw(.5, 0)--(.5, 1.5)--(1.5, 1.5)--(1.5,2);
\node[color=black] at (1.5,-0.5) {$4$};
\draw(1.5, 0)--(1.5, .5)--(2, .5);
\node[color=red] at (.5,2.5) {$1$};
\node[color=red] at (1.5,2.5) {$3$};
\node[color=red] at (2.5,1.5) {$2$};
\node[color=red] at (2.5,0.5) {$4$};
\end{tikzpicture}
\quad\raisebox{1 cm}{$\xrightarrow{f_1}$}\quad
\begin{tikzpicture}[x=1.5em,y=1.5em,thick,rounded corners, color = blue]
\draw[step=1,black, thin] (0,0) grid (2,2);
\node[color=black] at (-0.5,1.5) {$1$};
\draw(0, 1.5)--(.5,1.5)--(.5, 2);
\node[color=black] at (-0.5,.5) {$2$};
\draw(0, .5)--(1.5, .5)--(1.5,2);
\node[color=black] at (.5,-0.5) {$3$};
\draw(.5, 0)--(.5, 1.5)--(2, 1.5);
\node[color=black] at (1.5,-0.5) {$4$};
\draw(1.5, 0)--(1.5, .5)--(2, .5);
\node[color=red] at (.5,2.5) {$1$};
\node[color=red] at (1.5,2.5) {$3$};
\node[color=red] at (2.5,1.5) {$2$};
\node[color=red] at (2.5,0.5) {$4$};
\end{tikzpicture}
\quad\raisebox{1 cm}{$\xrightarrow{f_1}$}\quad
\begin{tikzpicture}[x=1.5em,y=1.5em,thick,rounded corners, color = blue]
\draw[step=1,black, thin] (0,0) grid (2,2);
\node[color=black] at (-0.5,1.5) {$1$};
\draw(0, 1.5)--(.5,1.5)--(.5, 2);
\node[color=black] at (-0.5,.5) {$2$};
\draw(0, .5)--(.5, .5)--(.5,1.5)--(2,1.5);
\node[color=black] at (.5,-0.5) {$3$};
\draw(.5, 0)--(.5, .5)--(1.5, .5)--(1.5, 2);
\node[color=black] at (1.5,-0.5) {$4$};
\draw(1.5, 0)--(1.5, .5)--(2, .5);
\node[color=red] at (.5,2.5) {$1$};
\node[color=red] at (1.5,2.5) {$3$};
\node[color=red] at (2.5,1.5) {$2$};
\node[color=red] at (2.5,0.5) {$4$};
\end{tikzpicture}
\quad\raisebox{1 cm}{$\xrightarrow{f_1}$}\quad
\begin{tikzpicture}[x=1.5em,y=1.5em,thick,rounded corners, color = blue]
\draw[step=1,black, thin] (0,0) grid (2,2);
\node[color=black] at (-0.5,1.5) {$1$};
\draw(0, 1.5)--(2, 1.5);
\node[color=black] at (-0.5,.5) {$2$};
\draw(0, .5)--(.5, .5)--(.5,2);
\node[color=black] at (.5,-0.5) {$3$};
\draw(.5, 0)--(.5, .5)--(1.5, .5)--(1.5, 2);
\node[color=black] at (1.5,-0.5) {$4$};
\draw(1.5, 0)--(1.5, .5)--(2, .5);
\node[color=red] at (.5,2.5) {$2$};
\node[color=red] at (1.5,2.5) {$3$};
\node[color=red] at (2.5,1.5) {$1$};
\node[color=red] at (2.5,0.5) {$4$};
\end{tikzpicture}.
$$
For this example, the proposition claims the following statements, which can be verified directly:
\begin{itemize}
\item[(a)] We have $u_0 = 1342$,
so $3=m+i$ (for $m=2$ and $i=1$) is a descent of $u_0$.
\item[(b)] We have $l=4$ and $u_l = 2314$,
so $i=1$ appears after $i+1=2$ in $u_l$.
\item[(c)] We have $u_1 = u_2 = u_3 = 1324 = u_0 s_{m+i} = s_i u_l$ (for $m=2$ and $i=1$).
\end{itemize}
\end{exa}

\begin{proof}[Proof of Proposition~\ref{P: string and Permutation}]
We prove the statement by induction on $m$.
Let $D_i'$ be the RPD obtained by removing
the last column of $D_i$.
We first consider $D_0$.
Row $i$ and row $i+1$ in the last column
of $D_j$ must be one of the following four
configurations:
$$
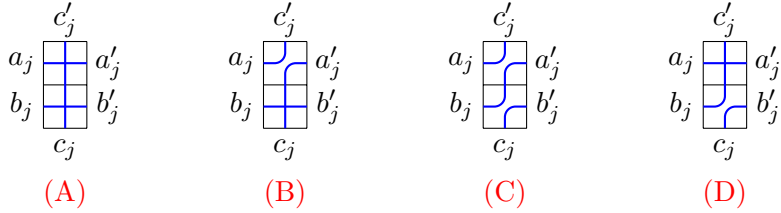
\begin{tikzpicture}[x=1.5em,y=1.5em,thick,rounded corners, color = blue]
\draw[step=1,black, thin] (0,0) grid (1,2);
\node[color=black] at (-0.5,1.5) {$a_j$};
\node[color=black] at (-0.5,.5) {$b_j$};
\node[color=black] at (.5,-.5) {$c_j$};
\node[color=black] at (1.5,1.5) {$a'_j$};
\node[color=black] at (1.5,.5) {$b'_j$};
\node[color=black] at (.5,2.5) {$c'_j$};
\node[color=red] at (.5,-1.5) {(A)};
\draw(.5, 0)--(.5, 2);
\draw(0, 1.5)--(1, 1.5);
\draw(0, 0.5)--(1, .5);
\end{tikzpicture}
\quad\quad\quad
\begin{tikzpicture}[x=1.5em,y=1.5em,thick,rounded corners, color = blue]
\draw[step=1,black, thin] (0,0) grid (1,2);
\node[color=black] at (-0.5,1.5) {$a_j$};
\node[color=black] at (-0.5,.5) {$b_j$};
\node[color=black] at (.5,-.5) {$c_j$};
\node[color=black] at (1.5,1.5) {$a'_j$};
\node[color=black] at (1.5,.5) {$b'_j$};
\node[color=black] at (.5,2.5) {$c'_j$};
\draw(0, 1.5)--(.5, 1.5)--(.5, 2);
\draw(0.5, 0)--(.5, 1.5)--(1, 1.5);
\draw(0, 0.5)--(1, .5);
\node[color=red] at (.5,-1.5) {(B)};
\end{tikzpicture}
\quad\quad\quad
\begin{tikzpicture}[x=1.5em,y=1.5em,thick,rounded corners, color = blue]
\draw[step=1,black, thin] (0,0) grid (1,2);
\node[color=black] at (-0.5,1.5) {$a_j$};
\node[color=black] at (-0.5,.5) {$b_j$};
\node[color=black] at (.5,-.5) {$c_j$};
\node[color=black] at (1.5,1.5) {$a'_j$};
\node[color=black] at (1.5,.5) {$b'_j$};
\node[color=black] at (.5,2.5) {$c'_j$};
\draw(0, 1.5)--(.5, 1.5)--(.5, 2);
\draw(0.5, 0)--(0.5, 0.5)--(1,.5);
\draw(0, 0.5)--(.5, .5)--(.5, 1.5)--(1, 1.5);
\node[color=red] at (.5,-1.5) {(C)};
\end{tikzpicture}
\quad\quad\quad
\begin{tikzpicture}[x=1.5em,y=1.5em,thick,rounded corners, color = blue]
\draw[step=1,black, thin] (0,0) grid (1,2);
\node[color=black] at (-0.5,1.5) {$a_j$};
\node[color=black] at (-0.5,.5) {$b_j$};
\node[color=black] at (.5,-.5) {$c_j$};
\node[color=black] at (1.5,1.5) {$a'_j$};
\node[color=black] at (1.5,.5) {$b'_j$};
\node[color=black] at (.5,2.5) {$c'_j$};
\node[color=red] at (.5,-1.5) {(D)};
\draw(0, 1.5)--(1, 1.5);
\draw(0.5, 0)--(0.5, 0.5)--(1,.5);
\draw(0, 0.5)--(.5, .5)--(.5, 2);
\end{tikzpicture}
$$
where $a_j,b_j,c_j,a_j',b_j',c_j'$ are the labels of the pipes.

We know the last column of $D_0$ must be 
in Case~(A) or Case~(B).
In Case~(A), $D_0', \dots, D_l'$ is an $i$-string,
and the last columns of $D_0, \dots, D_l$
are identical.
The result follows immediately by applying
the inductive hypothesis to the $i$-string 
$D_0', \dots, D_l'$.

Now assume the last column of $D_0$ 
is in Case~(B).
We have $a_0' = \max(b_0,c_0)$ and 
$b_0' = \min(b_0,c_0)$,
so $u_0(m + i) = a_0' > b_0' = u_0(m + i + 1)$.

Next, consider the edge case where $D_0'$ 
is an $i$-string of length $1$.
In this case, $l = 2$ while $D_1$ and $ D_2$ are obtained 
by changing the last column of $D_0$
to Cases~(C) and (D), respectively.
Thus, $a_j$, $b_j$, and $c_j$ remain unchanged 
for $j = 0, 1, 2$.

Suppose $D_0'$ contains only $\ptile$ tiles 
in rows $i$ and $i+1$.
Then $a_0 = i$, $b_0 = i+1$, and $c_0 > i+1$.
Clearly, $u_1 = u_0 s_{m+i}$ and 
$u_2 = s_i u_1$.
To show that $u_2^{-1}(i) > u_2^{-1}(i+1)$,
observe that pipe $i$ and pipe $i+1$ 
cross in row $i$ of the last column in $D_2$.
Otherwise, applying the inductive hypothesis 
to the $i$-string of $D_0'$,
we conclude that
(i) $a_0 > b_0$, and 
(ii) pipe $i$ and pipe $i+1$ have already crossed 
in $D_0'$.
From (ii), it follows that $u_2^{-1}(i) > u_2^{-1}(i+1)$.
From (i), we see that $u_2 = u_1$.
In addition, our claim that 
$u_1 \in\{ u_0, u_0 s_{m+i}\}$
is immediate.

Now, assume the $i$-string of $D_0'$
has more than one element,
so $D_l'$ is the sink of its $i$-string.
By the inductive hypothesis,
pipe $i$ and pipe $i+1$ already cross
in $D_l'$,
so $u_l^{-1}(i) > u_l^{-1}(i+1)$.
Furthermore, the top and right edge labels
of $D_{l-1}'$ agree with those of $D_l'$
or differ by swapping $i$ and $i+1$.
Since $D_l$ and $D_{l-1}$ agree in the 
last column, we have $u_{l-1} \in \{u_l,s_i u_l\}$.

Finally, we show that $u_1 = \cdots = u_{l-1}$
and $u_1 \in \{u_0, u_0 s_{m+i}\}$
by analyzing $\varepsilon_i(D_0')$.

\begin{itemize}
\item Case 1: suppose $\varepsilon_i(D_0') = 0$,
so $D_0' = D_1' = D_2'$.
By the inductive hypothesis, 
$a_j' > b_j'$ for $j = 0, 1, 2$.
We know that $D_1$ and $D_2$ are obtained by changing the 
last column of $D_0$ into Cases~(C) and (D), respectively.
Therefore, $u_1 \in \{u_0,u_0s_{m+i}\}$,
and we deduce that
$u_2 = u_1$ from the fact that $a_2 > b_2$.
Finally, we know that $D_2', \cdots, D_l'$
is an $i$-string.
By the inductive hypothesis, 
$c_k = c_2$ and 
$(a_k, b_k)$ is either $(a_2, b_2)$ or $(b_2, a_2)$
for $3 \leq k < l$.
Thus, $a_k' = \min(a_k, b_k) = \min(a_2, b_2) = a_2'$
and $b_k' = b_2'$ similarly,
so $u_2 = u_3 = \cdots = u_{l-1}$.

\item Case 2: suppose $\varepsilon_i(D_0') = 1$.
Then we know that $D_0' = D_1'$,
and $D_1', \dots, D_l'$
is an $i$-string without the root. 
The last column of $D_1$ is in Case~(C),
and then stays unchanged in $D_2, \cdots, D_l$.
Similar to Case 1, 
we have $u_1 = u_0$ or $u_0 s_{m + i}$.
By the inductive hypothesis, we have
$a_k = a_1$, $b_k = b_1$ and $c_k = c_1$
for $2 \leq k < l$,
so $u_1 = u_2 \cdots = u_{l-1}$.

\item Case 3: otherwise, we must have 
$\varepsilon_i(D_0') = 2$ since $D_0$
is a root.
The last columns of $D_0, \dots, D_l$ are the same,
and $D_0', \dots, D_l'$ is an $i$-string without the root
and the second element. 
By the inductive hypothesis,
$a_k = a_1$, $b_k = b_1$ and $c_k = c_1$
for $1 \leq k < l$,
so $u_0 = u_1 = \cdots = u_{l-1}$. \qedhere 
\end{itemize}
\end{proof}

Given $v \in S_n$, let $1^m \times v \in S_{m+n}$
denote the permutation that fixes each $i \in [m]$ and sends $m+i \mapsto m+v(i)$ for $i \in [n]$.
The group $S_n\times S_n$ acts on $S_{m+n}$ by $(u,v) : w\mapsto u\cdot w\cdot (1^m\times v^{-1})$.
The elements in the orbit of $w$ under this action 
are 
obtained from each other by permuting both the values $1,2, \dots, n$ and the last $n$ positions in the one-line representation of $w$.

\begin{defn}\label{sigma-def2}
For $D \in \RPD_{m,n}$,
let $\bar\sigma_D \in S_{m+n}$
 be the unique element in
 the $(S_n\times S_n)$-orbit of $\sigma_D$ satisfying   
$ \bar\sigma^{-1}_D(1)
 >
  \bar\sigma^{-1}_D(2)
  >\dots>
   \bar\sigma^{-1}_D(n)
$
and
$
\bar\sigma_D(m+1)
>
\bar\sigma_D(m+2)
>
\dots>
\bar\sigma_D(m+n).
$
\end{defn}

Recall that the $\sqgln$-crystal structure on $\RPD_{m,n}$ lets us view this set as a directed graph
whose edges correspond to the $f_i$ operators.

\begin{cor}\label{sigma-cor}
    If $D_1,D_2 \in \RPD_{m,n}$ are in the same connected component then $\bar\sigma_{D_1} = \bar\sigma_{D_2}$.
\end{cor}

\begin{proof}
It suffices to show that $\sigma_{D_1}$ and $\sigma_{D_2}$ are in the same $(S_n\times S_n)$-orbit when $D_1$ and $D_2$ belong to the same $i$-string in $\RPD_{m,n}$, and this is clear from part (c) of Proposition~\ref{P: string and Permutation}.
\end{proof}

The $\sqgln$-crystal weight of $D \in \RPD_{m, n}$ is the tuple
$\weight(D) \in \NN^n$
whose $i^\textsuperscript{th}$ entry is
the number of $\bumptile$ tiles in row $i$ of $D$,
and for convenience we define
\begin{equation}\textstyle\bumpchi_{m,n,w} := \sum_{D \in \RPD_{m,n}(w)} x^{\wt(D)}.\end{equation}
\begin{rem}
One might consider the alternative weight
$\crosswt(D)   \in \NN^n$ given by the sequence whose $i^\textsuperscript{th}$ entry is
the number of $\ptile$ tiles in row $i$ of $D$.
The related generating function
\begin{equation}\textstyle\crosschi_{m,n,w} := \sum_{D \in \RPD_{m,n}(w)} x^{\crosswt(D)}\end{equation} turns out to be closely connected to the \definition{Grothendieck polynomial} $\fG_w$ studied in \cite{Buch2002,FK,LenartTr},
which is a non-symmetric generalization of $G_w(x_1,x_2,\dots,x_n)$
from Section~\ref{pi-sect}.

Here, we define $\fG_w$ for $w \in S_\infty$ recursively by setting 
$\fG_{n\cdots 321} := x_{1}^{n-1} x_2^{n-2} \cdots x_{n-1}$ for each $n \in \PP$,
and then requiring that
$\partial_i^K(\fG_w) = \fG_{w s_i}$ for each $i \in \PP$
with $w(i) > w(i+1)$.
This definition omits signs that are often included in the literature, exactly as in Remark~\ref{beta-rem}.
The polynomials $\fG_w$ are significant in geometry as $K$-theory representatives for the structure sheaves of the Schubert varieties in the complete variety; see \cite[\S2]{Buch2002}.

 Assume that $w \in S_{m+n}$ satisfies $n<w(m+1)<\dots<w(m+n)$.
 Then the pipe dreams for $w^{-1}$ in the conventions of \cite{KnutsonMiller2004,KnutsonMiller} never have crossings past column $m$. 
 It therefore follows by examining
\cite[Cor.~5.4]{KnutsonMiller2004}
(and adjusting signs appropriately)  
that 
\begin{equation}
    \crosschi_{m,n,w} = \fG_w|_{x_{n+1}=x_{n+2}=\dots=0}
\end{equation}
Thus $\crosschi_{m,n,w}$ is just $\fG_w$ restricted to $x_1,x_2,\dots,x_n$.
Using a sign-adjusted version of the transition formula \cite[Cor.~3.10]{LenartTr}, one can prove by induction that 
in this case $\crosschi_{m,n,w}$ is $\fG$-positive in the sense of being a $\NN$-linear combination of Grothendieck polynomials.
 \end{rem}

\subsection{Demazure analogues and Lascoux positivity}

Suppose $\cC$ is a finite normal $\sqgln$-crystal and $w \in S_n$. We wish to define a subset $\cC_w\subseteq \cC$  analogous to the Demazure crystal $\cB_w$ discussed at the beginning of this section. 
The definition of $\fkD_i$ given in \eqref{fkD-eq} still makes sense for $\sqgln$-crystals, 
but the set 
$\fkD_{i_1}\fkD_{i_2}\cdots \fkD_{i_k}(\HW(\cC))$
now depends on the choice of reduced word $i_1i_2\cdots i_k$ for $w$. We shall instead define $\cC_w$ in a different way. 

We assume that $\cC\subseteq \RPD_{m,n}$ for some fixed $m \in \PP$. This restriction is fairly mild, since any finite normal $\sqgln$-crystal can be embedded in  $\RPD_{m,n}$ for some sufficiently large $m$. However, unlike the $\gl_n$-case, it is not clear that such an embedding is unique (up to permutation of isomorphic components).
   Therefore it will not be  immediately apparent how to upgrade our definition to a canonical one for arbitrary normal crystals.

Because $\cC$ is a set of rectangular pipe dreams, Definition~\ref{sigma-def2} associates to each $c \in \cC$ a  permutation $\bar \sigma_c \in S_{m+n}$,
which is constant on connected components by Corollary~\ref{sigma-cor}.

\begin{defn}
Write $\leq$ for the \definition{(strong) Bruhat order} on $S_n$.
Then define
\[
\cC_w = \left\{ c \in \cC :
\sigma_c = u \cdot \bar\sigma_c\cdot (1^m\times v)
\text{ for some $u,v \in S_n$ with $v\leq w$}\right\}.\]
\end{defn}

This set is not always a full subcrystal of $\cC$, although $\cC = \cC_{w_0}$ for $w_0=n\cdots 321 \in S_n$.
Thus $\cC_w$ is not necessarily a $\sqgln$-crystal. However, 
we define the character of $\cC_w$ in the usual way
to be the polynomial
$\ch(\cC_w) = \sum_{c \in \cC_w} x^{\weight(c)} \in \ZZ[x_1,x_2,\dots,x_n]$.

One reason for considering $\cC_w$ to be an interesting analogue of the Demazure crystal $\cB_w$ is because both sets share the nontrivial properties listed in the following result. 
\begin{thm}\label{Cw-thm}
    Suppose $i \in [n-1]$ and $w \in S_n$ are such that $w(i) <w(i+1)$. Then
    \[
    \cC_{ws_i} = \fkD_i (\cC_w)
    \quand
    \ch(\cC_{ws_i}) = \pi^K_i (\ch(\cC_w)).
    \]
\end{thm}

\begin{rem}\label{lasc-rem}
Because $\fkD_i^2 = \fkD_i$, this theorem implies that 
 for any $w \in S_n$ with a reduced word $i_1i_2\cdots i_k$,
we have these analogues of~\eqref{fkDfkD-eq} and~\eqref{EQ: Demazure character}:
\begin{equation}
\cC_w = \fkD_{i_1}\fkD_{i_2}\cdots \fkD_{i_k}(\cC_{\textrm{id}})
\quand
\ch(\cC_w) =  \pi^K_{i_1}\pi^K_{i_2}\cdots \pi^K_{i_k} \ch(\cC_\textrm{id})
.\end{equation}
\end{rem}

Before proving this result, let us mention our second motivation for considering $\cC_w$, which comes from the following conjectural positivity property.
A polynomial is \definition{Lascoux positive} if it is a $\NN$-linear combination of Lascoux polynomials $\fL_\alpha$.

\begin{conj}
    For any  $\sqgln$-crystal $\cC \subseteq \RPD_{m,n}$ it holds that $\ch(\cC_w)$ is Lascoux positive.
\end{conj}

To prove this conjecture, it suffices by Remark~\ref{lasc-rem} to show that $\ch(\cC_\textrm{id})$ is Lascoux positive.
We have verified this by computer for all $m,n\leq 5$ excluding $m=n=5$.
The special case of this conjecture when $w=w_0 := n\cdots 321 \in S_n$
is a consequence of Theorem~\ref{ch-thm} since $\cC_{w_0}= \cC$.
In this sense, the conjecture is a refinement of our main theorem.

To prove Theorem~\ref{Cw-thm}, we need one lemma.

\begin{lemma}
\label{L: Las op}
Let $D_0,D_1, \cdots, D_l$ be an $i$-string in a $\sqgln$-crystal.
Then
$$
\textstyle
\pi_i^K(x^{\weight(D_0)}) = \sum_{j = 0}^{l}  x^{\weight(D_j)}
\quand
\pi_i^K\left(\sum_{j = 0}^{l}  x^{\weight(D_j)}\right) = \sum_{j = 0}^{l}  x^{\weight(D_j)}.
$$
\end{lemma}

Informally, 
this lemma says that applying the operator $\pi_i^K$
to the source will yield the whole string, and 
applying this operator to the whole string 
will fix it. 
For instance, by applying the lemma to the $1$-string
in Example~\ref{E: string},
we have
$
\pi_1^K(x_1^2) = x_1^2 + x_1^2x_2 + x_1x_2 + x_1x_2^2 + x_2^2
$
and
$
\pi_1^K(x_1^2 + x_1^2x_2 + x_1x_2 + x_1x_2^2 + x_2^2) = x_1^2 + x_1^2x_2 + x_1x_2 + x_1x_2^2 + x_2^2.
$

\begin{proof}
We just prove the first equation, since the second immediately follows as $(\pi_i^K)^2=\pi_i^K$. 
As $\varphi_i'(D_{0}) = l$ and $\varepsilon_i'(D_{0}) = 0$ by hypothesis, 
Definition~\ref{sqgln-def}(a)
asserts that $\weight (D_{0})_i - \weight (D_{0})_{i+1} = \frac{l-0}{2} \geq 0$. Therefore we can write 
$ x^{\weight(D_l)} = f(x) x_i^M x_{i+1}^N 
$
for some monomial $f(x)$ not involving $x_i$ or $x_{i+1}$ and some integers $M \geq N\geq 0$ with $l =2(M-N)$.
Then we have
$$ 
\sum_{j=0}^l x^{\weight(D_j)}
=
f(x)\left(  x_i^{M}x_{i+1}^N + x_i^{M}x_{i+1}^{N+1} + x_i^{M-1}x_{i+1}^{N+1}+ x_i^{M-1}x_{i+1}^{N+2} + \dots + 
  x_i^{N+1}x_{i+1}^{M} + x_i^{N}x_{i+1}^{M} \right)
$$
by Definition~\ref{sqgln-def}(b), and one can check that the right side is $\pi_i^K(x^{\weight(D_0)})=f(x)\pi_i^K(x_i^M x_{i+1}^N)$. 
\end{proof}

We now prove Theorem~\ref{Cw-thm}
using use Proposition~\ref{P: string and Permutation} and Lemma~\ref{L: Las op}.

\begin{proof}[Proof of Theorem~\ref{Cw-thm}]

Let $D_0,D_1, \cdots, D_l$ be an $i$-string
in $\RPD_{m,n}$. 
Then define 
\[\textstyle A := \cC(w) \cap \{D_0, D_1,\cdots, D_l\}
\quand B := \cC(ws_i) \cap \{D_0,D_1, \cdots, D_l\}.\]
It suffices to show that if $A$ and $B$ are not both empty, then 
$B = \fkD_i(A)$ and
\begin{equation}\label{oth-eq} \textstyle
\pi_i^K\left(\sum_{D \in A} x^{\bumpwt(D)}\right)
= \sum_{D \in B} x^{\bumpwt(D)}.\end{equation} 

Let $\sigma = \sigma_{D_0}$
and
$\bar\sigma = \bar\sigma_{D_0}$.
Suppose $u, v \in S_n$ 
are such that $\sigma = u \cdot \overline{\sigma} \cdot (1^m\times v)$.
In view of Proposition~\ref{P: string and Permutation},
we know that each $\sigma_{D_i}$
for $ i \in [l]$
has the form $u' \cdot\overline{\sigma} \cdot(1^m \times v)$
or $u' \cdot\overline{\sigma} \cdot (1^m \times vs_i)$ for some $u' \in S_n$.
If we cannot find such any $v$ satisfying $v \leq w s_i$,
then $A = B = \emptyset$.
As there is nothing to show in this case, we may assume that $v\leq ws_i$.

We claim that we may further assume that $v(i) < v(i+1)$.
Suppose not. 
Let 
\[\tau = u^{-1} \cdot \sigma=\overline{\sigma} \cdot (1^m\times v).\]
By assumption, $\overline{\sigma}(m+1) > \cdots > \overline{\sigma}(m+n)$.
Since $v(i) > v(i+1)$,
we have 
\[(1^m\times v)(m+i) > (1^m\times v)(m+i+1) >m\]
so
$\tau(m+i) < \tau(m+i+1)$.
On the other hand,
$\sigma(m + i) > \sigma(m + i + 1)$ by Proposition~\ref{P: string and Permutation}.
The left action of $u^{-1}$ on $\sigma$
only permutes the values $1, \dots, n$
so $\tau(m + i)$
and $\tau(m + i + 1)$ must be in $[n]$.
Finally notice
if $t \in S_{m+n}$ is the transposition that 
swaps 
$\tau(m + i)$ and $\tau(m + i+1)$, then
$$
t \cdot \tau
= \tau\cdot s_{m + i}
\qquad\text{and therefore} 
\qquad
t \cdot u^{-1} \cdot \sigma
= \overline{\sigma} \cdot (1^m\times v s_i) .
$$
Thus, if we replace $u$ by $ut$
and $v$ by $vs_i$,
then $\sigma = u \cdot \overline{\sigma} \cdot (1^m\times v)$
still holds
but now $v(i) < v(i+1)$.
The Bruhat inequality $v \leq ws_i$ also still holds after this replacement, since if $v(i)>v(i+1)$ then $vs_i \leq v$.

Now, as we may assume that  $v(i) <v(i+1)$, $w(i) < w(i+1)$,
and $v \leq w s_i$,
the \definition{Lifting Property} of the Bruhat order (see \cite[\S2.2]{CCG})
implies that both $v \leq w$ and $v s_i \leq w s_i$.
The first inequality tells us that $D_0 \in A$ and the two inequalities together imply that $A \subseteq B$.

To finish the proof, we consider whether or not $\sigma_{D_1}$ belongs to $A$. There are two cases:
\begin{itemize}
    \item Suppose $D_1 \in A$,
so that $\sigma_{D_1} = u' \cdot \overline{\sigma}\cdot  (1^m\times v')$
for some $u', v' \in S_n$ with $v' \leq w$.
By Proposition~\ref{P: string and Permutation},
$\sigma_{D_2}, \dots, \sigma_{D_l}$
all agree with $\sigma_{D_1}$ or $s_i\sigma_{D_1}$,
so $\{D_1, \cdots, D_l\} \subseteq A \subseteq B$.

\item
Suppose instead that $D_1 \notin A$. Then $\sigma_{D_1} \neq \sigma$
so  Proposition~\ref{P: string and Permutation} tells us that
$\sigma_{D_1} = \sigma s_{m+i} = u \cdot \overline{\sigma} \cdot (1^m\times v s_{i})$.
Since $vs_i \leq w s_i$, it follows that
$D_1 \in B$. 
As it again holds that  $\sigma_{D_2}, \dots, \sigma_{D_l}$
all agree with $\sigma_{D_1}$ or $s_i\sigma_{D_1}$,
we deduce that $\{D_1, \cdots, D_l\}\subseteq  B$.
\end{itemize}
From this case analysis, we see that $B = \{D_0,D_1, \dots, D_l\}$
and that $A = B$ or $A=\{D_0\}$.
Clearly $A = \fkD_i(B)$ and the remaining identity \eqref{oth-eq} follows from Lemma~\ref{L: Las op}.
\end{proof}

\bibliographystyle{alpha}
\bibliography{citation}{}
\end{document}